\documentclass[11pt]{amsart}

\usepackage[a4paper,hmargin=2.6cm,vmargin=2.8cm]{geometry}
\usepackage{amssymb}
\usepackage{booktabs}
\usepackage{array}%
\usepackage{tikz}
\usepackage{adjustbox}%
\usepackage{needspace}%
\usepackage[hidelinks]{hyperref}

\newcommand{\whe}{\mathrel{\overset{\scriptscriptstyle\mathrm{we}}{\approx}}}
\theoremstyle{plain}
\newtheorem{theorem}{Theorem}[section]
\newtheorem{lemma}[theorem]{Lemma}
\newtheorem{corollary}[theorem]{Corollary}
\newtheorem{maintheorem}{Theorem}

\theoremstyle{definition}
\newtheorem{definition}[theorem]{Definition}
\theoremstyle{remark}
\newtheorem{remark}[theorem]{Remark}

\begin{document}

\title[Minimal finite models of wedges of spheres]{Minimal finite models of wedges of spheres with at most eight points}

\author[P. Das]{Ponaki Das}
\address{Department of Mathematics, North-Eastern Hill University, NEHU Campus, Shillong 793022, Meghalaya, India}
\email{ponaki.das20@gmail.com}

\author[S. M. Mawiong]{Sainkupar Marwein Mawiong}
\address{Department of Basic Sciences and Social Sciences, North-Eastern Hill University, NEHU Campus, Shillong 793022, Meghalaya, India}
\email{skupar@gmail.com}
\thanks{Corresponding author: Sainkupar Marwein Mawiong.}


\begin{abstract}
A minimal finite model of a topological space is a finite $T_0$-space of least cardinality that is weak homotopy equivalent to it. We determine, up to homeomorphism, the minimal finite models of eight wedges of one- and two-dimensional spheres. The spaces $S^{2}\vee S^{1}$ and $S^{2}\vee S^{2}$ have exactly two and three minimal finite models, respectively, all with seven points, whereas $S^{1}\vee S^{1}\vee S^{2}$ has exactly seven minimal finite models, $S^{1}\vee S^{1}\vee S^{1}\vee S^{2}$ has one, $S^{1}\vee S^{2}\vee S^{2}$ has six, $S^{2}\vee S^{2}\vee S^{2}$ has five, $S^{2}\vee S^{2}\vee S^{2}\vee S^{2}$ has three, and $S^{1}\vee S^{1}\vee S^{2}\vee S^{2}$ has two, all with eight points. To prove this we classify the connected finite $T_0$-spaces of height $2$ without beat points that have seven or eight points, and we identify their weak homotopy types by means of McCord's theorem and explicit collapse and suspension arguments for their two-dimensional order complexes. Together with the known results for spheres and wedges of circles, this determines, up to weak homotopy equivalence, all path-connected spaces whose minimal finite models have at most eight points.
\end{abstract}

\maketitle

\section{Introduction}\label{sec:intro}

Given a topological space, how few points can a finite space with the same weak homotopy type have? Barmak and Minian introduced the notion of a \emph{minimal finite model} to make this question precise~\cite{Barmak(2007)}. They showed that the minimal finite model of $S^n$ is unique and has $2n+2$ points, and they characterised the minimal finite models of wedges of circles~\cite{Barmak(2007),Barmak(2011)}. Cianci and Ottina's poset-splitting technique later turned the question into a working classification tool, powerful enough to determine the minimal finite models of the torus, the Klein bottle and the real projective plane~\cite{Cianci-Ottina(2018)}. The present paper treats wedges of spheres beyond these cases, in particular wedges of spheres of different dimensions. The homotopy theory of such wedges is classical (their homotopy groups were described by Hilton~\cite{Hilton(1955)}), but it says nothing about how economically a finite space can realise them.

Models are counted up to homeomorphism, not merely up to weak homotopy type. Since a finite poset and its opposite have the same order complex, the opposite of a minimal finite model of a space is again a minimal finite model of that space; it is counted as a separate model precisely when it is not homeomorphic to the original (Remark~\ref{rem:count}). Our first main result is the following classification.

\Needspace{15\baselineskip}
\begin{maintheorem}\label{thm:A}
	The minimal finite models of the following eight wedges of spheres have the stated numbers of points, and up to homeomorphism there are exactly the stated numbers of them.
	\begin{center}
		\begin{tabular}{lcc}
			\toprule
			Space & Points & Minimal finite models\\
			\midrule
			$S^2\vee S^1$ & $7$ & $2$ (a dual pair)\\
			$S^2\vee S^2$ & $7$ & $3$\\
			$S^1\vee S^1\vee S^2$ & $8$ & $7$\\
			$S^1\vee S^1\vee S^1\vee S^2$ & $8$ & $1$\\
			$S^1\vee S^2\vee S^2$ & $8$ & $6$\\
			$S^2\vee S^2\vee S^2$ & $8$ & $5$\\
			$S^2\vee S^2\vee S^2\vee S^2$ & $8$ & $3$\\
			$S^1\vee S^1\vee S^2\vee S^2$ & $8$ & $2$ (a dual pair)\\
			\bottomrule
		\end{tabular}
	\end{center}
\end{maintheorem}

Theorem~\ref{thm:A} combines Theorems~\ref{thm:seven} and~\ref{thm:eight}, where all the models are displayed. Together with the classical results on spheres and wedges of circles, it yields the complete list, up to weak homotopy equivalence, of path-connected spaces whose minimal finite models have at most eight points. In the statement below, $Y\whe Z$ means that $Y$ and $Z$ are weak homotopy equivalent (see Section~\ref{sec:prelim}).

\begin{maintheorem}\label{thm:B}
	Suppose that the minimal finite models of a path-connected space $Y$ have $n\le8$ points. Then $n\in\{1,4,5,6,7,8\}$, and:
	\begin{enumerate}
		\item $n=1$ if and only if $Y$ is weakly contractible;
		\item $n=4$ if and only if $Y\whe S^1$;
		\item $n=5$ if and only if $Y\whe S^1\vee S^1$;
		\item $n=6$ if and only if $Y\whe S^2$, $\bigvee_3S^1$ or $\bigvee_4S^1$;
		\item $n=7$ if and only if $Y\whe S^2\vee S^1$, $S^2\vee S^2$, $\bigvee_5S^1$ or $\bigvee_6S^1$;
		\item $n=8$ if and only if $Y\whe S^3$, $\bigvee_7S^1$, $\bigvee_8S^1$, $\bigvee_9S^1$, or one of the six eight-point wedges of Theorem~\ref{thm:A}.
	\end{enumerate}
\end{maintheorem}

Theorem~\ref{thm:B} is proved as Corollary~\ref{cor:small}.

Table~\ref{tab:known} places our results among the minimal finite models known so far. In it, ``points'' is the cardinality of the minimal finite models, the third column describes them up to homeomorphism, giving their number where it is known, and the height is that of every minimal finite model of the space. A finite connected graph with first Betti number $n$ is homotopy equivalent to $\bigvee_nS^1$, so the row for wedges of circles covers all finite graphs; there $N(n)=\min\{\,i+j\mid (i-1)(j-1)\ge n\,\}$ (Corollary~\ref{cor:barmak82}), so that $N(n)=4,5,6,6,7,7,8,8,8$ for $n=1,\dots,9$ and $N(n)\ge9$ for $n\ge10$.

\begin{table}[htbp]
	\centering
	\caption{Known minimal finite models}
	\label{tab:known}
	\renewcommand{\arraystretch}{1.25}
	\begin{tabular}{>{\raggedright\arraybackslash}p{3.6cm} >{\centering\arraybackslash}p{1.5cm} >{\raggedright\arraybackslash}p{5.0cm} >{\centering\arraybackslash}p{1.2cm} >{\raggedright\arraybackslash}p{2.2cm}}
		\toprule
		Space & Points & Minimal finite models up to homeomorphism & Height & Source\\
		\midrule
		point & $1$ & $1$ & $0$ & \cite[Sec.~3.1]{Barmak(2011)}\\
		$S^n$, $n\ge1$ & $2n+2$ & $1$, namely $\mathbb{S}^n(S^0)$ & $n$ & \cite{Barmak(2007)}\\
		$\bigvee_nS^1$, $n\ge1$ (finite graphs) & $N(n)$ & the spaces of height $1$ with $N(n)$ points and $N(n)+n-1$ edges in the Hasse diagram & $1$ & \cite{Barmak(2007)}\\
		\midrule
		$S^2\vee S^1$ & $7$ & $2$ (a dual pair) & $2$ & Theorem~\ref{thm:A}\\
		$S^2\vee S^2$ & $7$ & $3$ & $2$ & Theorem~\ref{thm:A}\\
		$S^1\vee S^1\vee S^2$ & $8$ & $7$ & $2$ & Theorem~\ref{thm:A}\\
		$S^1\vee S^1\vee S^1\vee S^2$ & $8$ & $1$ & $2$ & Theorem~\ref{thm:A}\\
		$S^1\vee S^2\vee S^2$ & $8$ & $6$ & $2$ & Theorem~\ref{thm:A}\\
		$S^2\vee S^2\vee S^2$ & $8$ & $5$ & $2$ & Theorem~\ref{thm:A}\\
		$S^2\vee S^2\vee S^2\vee S^2$ & $8$ & $3$ & $2$ & Theorem~\ref{thm:A}\\
		$S^1\vee S^1\vee S^2\vee S^2$ & $8$ & $2$ (a dual pair) & $2$ & Theorem~\ref{thm:A}\\
		\midrule
		$S^n\vee S^{n-1}\vee S^{n-1}$, $n\ge2$ & $2n+4$ & $7$ for $n=2$ (Theorem~\ref{thm:A}) & $n$ & \cite{MosqueraLois(2026)}\\
		\midrule
		$\mathbb{R}P^2$ & $13$ & $2$ (a dual pair) & $2$ & \cite{Cianci-Ottina(2018)}\\
		torus $T^2$ & $16$ & $2$ & $2$ & \cite{Cianci-Ottina(2018)}\\
		Klein bottle & $16$ & $2$ (a dual pair) & $2$ & \cite{Cianci-Ottina(2018)}\\
		\bottomrule
	\end{tabular}
\end{table}

The heights in the first three rows of Table~\ref{tab:known} follow from Theorems~\ref{thm:barmak8} and~\ref{thm:barmak81}, those for $\mathbb{R}P^2$, $T^2$ and the Klein bottle from~\cite{Cianci-Ottina(2018)}, and those for the eight wedges of Theorem~\ref{thm:A} from Remark~\ref{rem:ht2} and Lemma~\ref{lem:lb8}. For $S^n\vee S^{n-1}\vee S^{n-1}$, a minimal finite model $X$ has $2n+4$ points, so $\mathrm{h}(X)\le n+1$ by Theorem~\ref{thm:barmak7}, and $\mathrm{h}(X)\ge n$, since $H_n(X)\neq0$; moreover $\mathrm{h}(X)\neq n+1$, since otherwise $\#X=2\mathrm{h}(X)+2$ and $X$ would be homeomorphic to $\mathbb{S}^{n+1}(S^0)$, whose $n$-th homology vanishes. For comparison, the smallest weakly contractible non-contractible finite spaces have $9$ points~\cite{Cianci-Ottina(2020)}, whereas the minimal finite model of a weakly contractible space is a single point.

Our method works directly with the combinatorics of height-$2$ posets. We combine a combinatorial criterion for a poset to be a core, closed-form Betti-number formulas, and explicit collapsing and suspension arguments for identifying weak homotopy types (all recorded in Section~\ref{sec:prelim}). The only input we need from the poset-splitting theory of~\cite{Cianci-Ottina(2018)} is an elementary constraint on intermediate points (Lemma~\ref{lem:co}). Candidate models are found with the help of regular CW complexes, but their minimality is proved entirely by finite-space methods: an arbitrary model is reduced to a core, small cores are controlled by Barmak's lower bound (Theorem~\ref{thm:barmak7}) and by the characterisation of the minimal finite models of wedges of circles (Theorem~\ref{thm:barmak81}), and the remaining connected cores of height $2$ with seven and with eight points are classified completely.

After an earlier version of this paper appeared on the arXiv, Mosquera-Lois~\cite{MosqueraLois(2026)} proved, by an independent argument and in every dimension $n\ge2$, that the minimal finite models of $S^n\vee S^{n-1}\vee S^{n-1}$ have $2n+4$ points; for $n=2$ this is the eight-point count for $S^1\vee S^1\vee S^2$ obtained here, which he cites from that version. His example of a weak homotopy equivalence between minimal finite models that is not a homotopy equivalence consists of two of the seven minimal finite models of $S^1\vee S^1\vee S^2$ found here. The classification up to homeomorphism and the counts in Theorem~\ref{thm:A}, as well as Theorem~\ref{thm:B}, are not considered in~\cite{MosqueraLois(2026)}; see Remark~\ref{rem:ML}.

We now outline the paper. Section~\ref{sec:prelim} fixes notation and gathers the tools used throughout: the size and rigidity constraints on minimal finite models of spheres and of wedges of circles, the constraint on intermediate points, the combinatorial core criterion, and the closed-form Betti-number formulas used to read off weak homotopy types from a poset's covering relations. Section~\ref{sec:main} constructs small models from regular CW complexes and proves the classifications. Section~\ref{sec:concl} discusses connections with digital topology and discrete Morse theory and what remains open beyond eight points, and Appendix~\ref{app:nonmin} lists the eight-point cores of height $2$ that are not minimal finite models.

\section{Preliminaries}\label{sec:prelim}

This section collects the basic notions and results on finite spaces that will be used throughout the paper. We mostly follow the notation of Barmak~\cite{Barmak(2011)}; our notation for height-$2$ posets is adapted from Cianci--Ottina~\cite{Cianci-Ottina(2018)}. For general background on finite spaces we refer to~\cite{Barmak(2011),May(2003)}, and for algebraic and general topology to~\cite{Greenberg-Harper(1981),Hatcher(2002),Munkres(1984),Munkres(2000)}.

\begin{definition}[Minimal finite model~{\cite[Def.~3.1.1]{Barmak(2011)}}]\label{def:mfm}
	Let $X$ be a topological space.
	A finite $T_0$-space $Y$ is a \emph{finite model} of $X$ if $Y$ is weak homotopy equivalent to $X$.
	It is a \emph{minimal finite model} of $X$ if $Y$ has the smallest possible cardinality among all such models.
\end{definition}

Here a continuous map $f\colon Z\to W$ is a \emph{weak homotopy equivalence} if it induces a bijection $f_\ast\colon\pi_0(Z)\to\pi_0(W)$ and isomorphisms $f_\ast\colon\pi_n(Z,z_0)\to\pi_n(W,f(z_0))$ for all $n\ge1$ and every base point $z_0\in Z$. We write $Z\whe W$ when $Z$ and $W$ are joined by a finite zig-zag of weak homotopy equivalences, and $Z\simeq W$ when they are homotopy equivalent.

\begin{remark}[Homology of wedges of spheres]\label{rem:wedgeinv}
	A weak homotopy equivalence induces isomorphisms in singular homology~\cite[Prop.~4.21]{Hatcher(2002)}, so the Betti numbers $\beta_i=\operatorname{rank}H_i$ are invariants of $\whe$. If $W=S^{n_1}\vee\cdots\vee S^{n_r}$ with all $n_j\ge1$, then for every $k\ge1$ the group $H_k(W)$ is free abelian of rank $\#\{j\mid n_j=k\}$~\cite[Cor.~2.25]{Hatcher(2002)}. Consequently, two such wedges are weak homotopy equivalent if and only if they have the same number of spheres of each dimension. In particular, $(\bigvee_rS^1)\vee(\bigvee_sS^2)$ has $(\beta_1,\beta_2)=(r,s)$: the eight wedges of Theorem~\ref{thm:A}, in the order listed there, have $(\beta_1,\beta_2)=(1,1)$, $(0,2)$, $(2,1)$, $(3,1)$, $(1,2)$, $(0,3)$, $(0,4)$ and $(2,2)$, while $S^2$ and $\bigvee_kS^1$ have $(0,1)$ and $(k,0)$, respectively.
\end{remark}

\subsection{Finite spaces, posets and cores}\label{subsec:basics}

A finite $T_0$-space $X$ carries the \emph{specialisation order}: $x\le y$ if and only if every open set that contains $y$ also contains $x$. This gives a bijective correspondence between finite $T_0$-spaces and finite posets, under which the continuous maps are exactly the order-preserving maps and the homeomorphisms are exactly the order isomorphisms~\cite[Chap.~1]{Barmak(2011)}. We freely identify a finite $T_0$-space with its poset, and for a finite set $S$ we write $\#S$ or $|S|$ for its cardinality. For a point $x$ of a finite $T_0$-space $X$ we write
\begin{gather*}
U_{x} = \{\, y \in X \mid y \le x \,\}, \qquad
F_{x} = \{\, y \in X \mid y \ge x \,\},\\
\hat{U}_{x} = U_{x} \setminus \{ x \}, \qquad
\hat{F}_{x} = F_{x} \setminus \{ x \};
\end{gather*}
thus $U_x$ is the minimal open set containing $x$ and $F_x$ is the closure of $\{x\}$. We write $x<y$ when $x\le y$ and $x\neq y$, and $x\prec y$, read ``$y$ covers $x$'', when $x<y$ and no point $z$ satisfies $x<z<y$. The \emph{Hasse diagram} $\mathcal{H}(X)$ is the graph with vertex set $X$ whose edges, forming the set $E(\mathcal{H}(X))$, are the covering pairs $x\prec y$. Two points are \emph{incomparable} if neither lies below the other, and an \emph{antichain} is a set of pairwise incomparable points. We write $\mathrm{mxl}(X)$ and $\mathrm{mnl}(X)$ for the sets of maximal and minimal points of $X$, and
\[
B_X = X\setminus\big(\mathrm{mxl}(X)\cup\mathrm{mnl}(X)\big)
\]
for the set of \emph{intermediate} points, those that are neither maximal nor minimal.

The simplicial complex whose simplices are the non-empty chains of $X$ is called the \emph{order complex} $\mathcal{K}(X)$; in the other direction, the \emph{face poset} $\mathcal{X}(K)$ of a simplicial or regular CW complex $K$ is the set of its cells ordered by the face relation ($e\le e'$ if and only if $e\subseteq\overline{e'}$). The \emph{length} of a chain $x_0<x_1<\cdots<x_k$ is $k$, the number of strict inequalities in it, and the \emph{height} $\mathrm{h}(X)$ is the maximum length of a chain in $X$; hence $\mathrm{h}(X)=\dim\mathcal{K}(X)$. The \emph{opposite} (or \emph{dual}) poset $X^{\mathrm{op}}$ has the same underlying set with the reversed order. Since $X$ and $X^{\mathrm{op}}$ have the same chains, $\mathcal{K}(X^{\mathrm{op}})=\mathcal{K}(X)$.

The \emph{non-Hausdorff join} $X\circledast Y$ of finite $T_0$-spaces $X$ and $Y$ is their disjoint union, with $x<y$ for every $x\in X$ and $y\in Y$ and the given orders on $X$ and $Y$. The \emph{non-Hausdorff suspension} of $X$ is $\mathbb{S}(X)=X\circledast S^0$, where $S^0$ denotes the two-point antichain~\cite[Def.~2.7.1]{Barmak(2011)}.

A point $x\in X$ is a \emph{down beat point} if $\hat U_x$ has a maximum, and an \emph{up beat point} if $\hat F_x$ has a minimum; a \emph{beat point} is a point that is either~\cite[Def.~1.3.3]{Barmak(2011)}. A finite $T_0$-space without beat points is called a \emph{minimal finite space}, and such a space is also said to \emph{be a core}. In addition, we use the following conventions.
\begin{enumerate}
	\item If $\mathrm{h}(X)=2$, a covering pair $c\prec a$ with $c\in\mathrm{mnl}(X)$ and $a\in\mathrm{mxl}(X)$ is called a \emph{direct cover}. Equivalently, $c<a$ and no intermediate point lies between $c$ and $a$.
	\item For a vertex $v$ of $\mathcal{K}(X)$, the \emph{closed star} of $v$ is $\overline{\mathrm{st}}(v)=\mathcal{K}(U_v\cup F_v)$, the order complex of the set of points comparable with $v$. It is a simplicial cone with apex $v$, hence contractible.
	\item $\mathbb{S}^{h}(S^0)=\mathbb{S}(\mathbb{S}(\cdots\mathbb{S}(S^0)\cdots))$ denotes the $h$-fold non-Hausdorff suspension of $S^0$. It is a core with $2h+2$ points.
\end{enumerate}

\begin{remark}\label{rem:warning}
	The terms \emph{minimal finite space} (a core) and \emph{minimal finite model} (Definition~\ref{def:mfm}, a smallest-cardinality finite model of a given space) are distinct. A minimal finite model is always a minimal finite space~\cite[Rem.~3.1.2]{Barmak(2011)}. A minimal finite space, however, need not be a minimal finite model of the space it represents, since it may have more points than necessary; for example, none of the eight-point cores of Table~\ref{tab:nonmin} is a minimal finite model.
\end{remark}

We record two classical facts as a single numbered statement, since they are used repeatedly below: removing a beat point is a strong deformation retraction, and a finite $T_0$-space is weak homotopy equivalent to the geometric realisation of its order complex.

\begin{theorem}[Stong~\cite{Stong(1966)}; McCord~\cite{McCord(1966)}]\label{thm:mccord}
	Let $X$ be a finite $T_0$-space.
	\begin{enumerate}
		\item If $x$ is a beat point of $X$, then $X\setminus\{x\}$ is a strong deformation retract of $X$~\cite[Prop.~1.3.4]{Barmak(2011)}. Consequently $X$ strong deformation retracts onto a minimal finite space, its core~\cite[Thm.~2]{Stong(1966)}.
		\item There is a weak homotopy equivalence $|\mathcal{K}(X)|\to X$~\cite[Thm.~2]{McCord(1966)}. In particular $X\whe X^{\mathrm{op}}$, since $\mathcal{K}(X^{\mathrm{op}})=\mathcal{K}(X)$.
	\end{enumerate}
\end{theorem}

The next results provide sharp constraints on the size of minimal finite models of spheres and of wedges of circles. Recall that $\mathbb{S}^{h}(S^0)$ is the minimal finite model of $S^{h}$~\cite[Prop.~3.1.4 and Cor.~3.2.3]{Barmak(2011)}.

\begin{theorem}[{Barmak~\cite[Thm.~3.2.1]{Barmak(2011)}; see also Barmak--Minian~\cite[Thm.~2.12]{Barmak(2007)}}]\label{thm:barmak7}
	Let $X\neq \ast$ be a minimal finite space. Then $X$ has at least $2\mathrm{h}(X)+2$ points. Moreover, if $X$ has exactly $2\mathrm{h}(X)+2$ points, then it is homeomorphic to $\mathbb{S}^{\mathrm{h}(X)}(S^0)$.
\end{theorem}
Throughout, $\mathrm{h}(X)$ denotes the maximum \emph{length} of a chain. In~\cite{Barmak(2007)} this theorem and the one quoted in Theorem~\ref{thm:barmak81} are stated with $\mathrm{h}(X)$ denoting the number of \emph{points} of a longest chain, which is one more than our $\mathrm{h}(X)$.
\begin{theorem}[{Barmak~\cite[Thm.~3.2.2]{Barmak(2011)}; Barmak--Minian~\cite[Thm.~2.13]{Barmak(2007)}}]\label{thm:barmak8}
	Any finite $T_0$-space with the same homotopy groups as $S^n$ has at least $2n+2$ points, and, up to homeomorphism, $\mathbb{S}^{n}(S^0)$ is the unique such space with exactly $2n+2$ points.
\end{theorem}
\begin{theorem}[{Barmak~\cite[Thm.~3.3.7]{Barmak(2011)}; Barmak--Minian~\cite[Thm.~4.7]{Barmak(2007)}}]\label{thm:barmak81}
	Let $n\in \mathbb{N}$. A finite $T_0$-space $X$ is a minimal finite model of $\bigvee_{i=1}^{n} S^{1}$ if and only if $\mathrm{h}(X)=1$,
	\[
	\# X=\min\{\,i+j\mid (i-1)(j-1)\ge n\,\}
	\quad\text{and}\quad
	\#E(\mathcal{H}(X))=\#X+n-1,
	\]
	where the minimum runs over integers $i,j\ge 1$. In particular, every minimal finite model of $\bigvee_{i=1}^{n}S^1$ has height $1$~\cite[Cor.~3.3.6]{Barmak(2011)}.
\end{theorem}
\begin{corollary}[{Barmak~\cite[Thm.~3.3.7 and Cor.~3.3.8]{Barmak(2011)}; Barmak--Minian~\cite[Rem.~4.8]{Barmak(2007)}}]\label{cor:barmak82}
	The cardinality of a minimal finite model of $\bigvee_{i=1}^{n} S^{1}$ equals
	\[
	\min\{\,i+j\mid (i-1)(j-1)\ge n\,\}.
	\]
\end{corollary}
For $n=1,2,\dots,9$ this minimum equals $4,5,6,6,7,7,8,8,8$, respectively, and it is at least $9$ for $n\ge10$.

\begin{remark}\label{rem:mxlmnl}
	A connected finite $T_0$-space with a single maximal point has a maximum and is therefore contractible (the identity map is comparable to, and hence homotopic to, the constant map at that maximum), and dually for a single minimal point~\cite[Rem.~1.2.8]{Barmak(2011)}. Hence every non-contractible connected finite $T_0$-space, and in particular every core with more than one point, has at least two maximal and at least two minimal points.
\end{remark}

\begin{remark}\label{rem:ht2}
	For a finite poset $X$ we have $\mathrm{h}(X)=\dim\mathcal{K}(X)$, so a finite model of a space with nontrivial second homology has height at least $2$. A finite model may have height greater than $2$; however, by Theorem~\ref{thm:barmak7} a core of height $h\ge 3$ has at least $2h+2\ge 8$ points. Consequently a \emph{seven-point} core that models $S^2\vee S^1$ or $S^2\vee S^2$ has height exactly $2$. Similarly, a core on at most $8$ points modelling any of the two-dimensional wedges considered here has height $2$, unless it has exactly $8$ points and is the minimal finite model $\mathbb{S}^3(S^0)$ of $S^3$ (Theorem~\ref{thm:barmak7}); this case is excluded, since $S^3$ is not such a wedge.
\end{remark}

To analyse the height-$2$ cases we use the following elementary constraint on intermediate points, a special case of~\cite[Rem.~4.1]{Cianci-Ottina(2018)} (see also~\cite[Prop.~2.3(b)]{Cianci-Ottina(2020)}); in Stong's language it goes back to~\cite[Prop.~12]{Stong(1966)}.

\begin{lemma}\label{lem:co}
	Let $X$ be a finite $T_0$-space of height $2$ that is a core, and let $b\in B_X$. Then
	\[
	\#(\hat F_b\cap \mathrm{mxl}(X))\ge 2
	\qquad\text{and}\qquad
	\#(\hat U_b\cap \mathrm{mnl}(X))\ge 2 .
	\]
\end{lemma}

\begin{proof}
	Since $\mathrm{h}(X)=2$, every element strictly above $b$ is maximal, so $\hat F_b\subseteq\mathrm{mxl}(X)$. If $\#\hat F_b\le 1$, then either $\hat F_b=\varnothing$, so that $b$ is maximal, contradicting $b\in B_X$; or $\hat F_b=\{a\}$ has a minimum, making $b$ an up beat point and contradicting that $X$ is a core. Hence $\#(\hat F_b\cap\mathrm{mxl}(X))\ge 2$, and dually $\#(\hat U_b\cap\mathrm{mnl}(X))\ge 2$.
\end{proof}

\subsection{Auxiliary lemmas}\label{subsec:aux}

We record the structural lemmas used in the height-$2$ analysis, by which the weak homotopy type of a core is identified from its order complex; the order complexes occurring there are two-dimensional.

\begin{remark}\label{rem:method}
	For a finite $T_0$-space $X$ of height $2$, McCord's theorem (Theorem~\ref{thm:mccord}(2)) gives $X\whe|\mathcal{K}(X)|$ with $\dim\mathcal{K}(X)=2$. Throughout, whether a poset is a core is decided by the combinatorial criterion of Lemma~\ref{lem:crit}, which makes no reference to homotopy type. The weak homotopy type of a core is then identified by explicit structural lemmas: Lemma~\ref{lem:onemid} for one intermediate point, Lemma~\ref{lem:tripartite} for complete tripartite posets, Lemma~\ref{lem:pifree} for at most two maximal, at most two minimal or exactly two intermediate points, and Lemmas~\ref{lem:twomax} and~\ref{lem:twomid} for two maximal or two intermediate points. We do not invoke the classical fact that a connected two-dimensional CW complex with free fundamental group is homotopy equivalent to a wedge of circles and $2$-spheres (see~\cite{HMS(1993)}): every wedge decomposition below comes from an explicit collapse or suspension structure. By Theorem~\ref{thm:mccord}(2), $X$ and $X^{\mathrm{op}}$ always have the same weak homotopy type, so it suffices to treat one representative of each dual pair.
\end{remark}

The following observation removes the single-intermediate-point case from all that follows.

\begin{lemma}\label{lem:onemid}
	Let $X$ be a connected finite $T_0$-space of height $2$ with $\#B_X=1$. Then $\mathcal{K}(X)$ collapses to a graph; consequently $H_2(\mathcal{K}(X))=0$ and $X$ is weak homotopy equivalent to a wedge of circles. In particular no such $X$ is a finite model of a space with nontrivial second homology.
\end{lemma}

\begin{proof}
	Write $B_X=\{b\}$. Since $\mathrm{h}(X)=2$, every $2$-simplex of $\mathcal{K}(X)$ is a chain $c<b'<a$ with $b'\in B_X$, so $b'=b$. Thus every triangle has the form $\{c,b,a\}$ with $c\in\hat U_b\cap\mathrm{mnl}(X)$ and $a\in\hat F_b\cap\mathrm{mxl}(X)$. For such a triangle, the edge $\{c,a\}$ lies in no other triangle, since any triangle containing both $c$ and $a$ would be $\{c,b',a\}$ with $b'\in B_X$, forcing $b'=b$. Hence each $\{c,a\}$ is a free face. Distinct triangles have distinct free faces, since the pair $(c,a)$ determines the triangle $\{c,b,a\}$; so the elementary collapses $(\{c,a\},\{c,b,a\})$ involve pairwise distinct simplices and may be performed in any order. Doing so reduces $\mathcal{K}(X)$ to a subcomplex of its $1$-skeleton, that is, to a graph $G$. Collapses preserve homotopy type, so $\mathcal{K}(X)\simeq G$, and $G$ is connected because $X$ is. As $G$ is a connected graph, $\mathcal{K}(X)\simeq\bigvee_r S^1$ with $r=\operatorname{rank}H_1(G)=\operatorname{rank}H_1(\mathcal{K}(X))$; in particular $H_2(\mathcal{K}(X))=0$.
\end{proof}

The next lemma converts ``$X$ is a core'' into a purely combinatorial condition on covers; it is used throughout the classification in place of case-by-case beat-point checks. Here an \emph{immediate predecessor} of $x$ is a point $y$ with $y\prec x$, and an \emph{immediate successor} of $x$ is a point $y$ with $x\prec y$.

\begin{lemma}[Core criterion]\label{lem:crit}
	A connected finite $T_0$-space $X$ of height $2$ is a core if and only if no point has exactly one immediate predecessor and no point has exactly one immediate successor. Equivalently, with the levels $\mathrm{mnl}(X)$, $B_X$, $\mathrm{mxl}(X)$:
	\begin{enumerate}
		\item every $b\in B_X$ covers at least two minimal points and is covered by at least two maximal points;
		\item every minimal point has at least two immediate successors (intermediate points above it, or maximal points covering it directly);
		\item every maximal point has at least two immediate predecessors (intermediate points below it, or minimal points that it covers directly).
	\end{enumerate}
\end{lemma}

\begin{proof}
	By~\cite[Def.~1.3.3]{Barmak(2011)}, a point $x$ is a down beat point if and only if it covers exactly one point, and an up beat point if and only if it is covered by exactly one point. Since $X$ is a core if and only if it has no beat points, this gives the first formulation. For the second, note that the relevant sides are never empty. A minimal point has no predecessors and, since $X$ is connected with more than one point, at least one successor; so the condition on it is exactly (2). Dually, the condition on a maximal point is (3). An intermediate point $b$ has a minimal point below it and a maximal point above it, and, as $\mathrm{h}(X)=2$, the points below $b$ are minimal and those above $b$ are maximal. So the immediate predecessors of $b$ are the minimal points below it, its immediate successors are the maximal points above it, and the condition on $b$ is (1). Finally, the immediate successors of a minimal point $c$ are the intermediate points above $c$ (nothing lies strictly between them) together with the maximal points covering $c$ directly, and dually for maximal points.
\end{proof}

The next two lemmas isolate the homotopy-theoretic steps used repeatedly below: the union of two cones along a graph is a suspension, and attaching cells of dimension at most one to a disjoint union of wedges of circles and $2$-spheres yields again such a wedge.

\begin{lemma}[Union of two cones]\label{lem:twocones}
	Let $K$ be a connected finite CW complex that is the union of two contractible subcomplexes $Q_1,Q_2$ whose intersection $L=Q_1\cap Q_2$ has dimension at most one. Let $\kappa$ be the number of connected components of $L$ and $\ell=\operatorname{rank}H_1(L)$. Then $L\neq\varnothing$ and
	\[
	K\;\simeq\;\Sigma L\;\simeq\;\Big(\bigvee_{\kappa-1}S^1\Big)\vee\Big(\bigvee_{\ell}S^2\Big),
	\]
	where $\Sigma L$ denotes the unreduced suspension of $L$.
\end{lemma}

\begin{proof}
	If $L$ were empty, $K$ would be the disjoint union of the non-empty sets $Q_1$ and $Q_2$, contradicting connectedness; so $L\neq\varnothing$.
	
	\emph{Step 1: $K\simeq\Sigma L$.} Since $Q_1$ is a contractible subcomplex, the quotient map $K\to K/Q_1$ is a homotopy equivalence~\cite[Prop.~0.17]{Hatcher(2002)}; and $K/Q_1$ is homeomorphic to $Q_2/L$, because $K\setminus Q_1=Q_2\setminus L$ makes the induced map $Q_2/L\to K/Q_1$ a continuous bijection between compact Hausdorff spaces. Let $Q_2\cup CL$ be obtained from $Q_2$ by attaching the cone on $L$ along the inclusion $L\hookrightarrow Q_2$. Collapsing the contractible subcomplex $CL$ gives $Q_2\cup CL\simeq Q_2/L$~\cite[Prop.~0.17]{Hatcher(2002)}, while the attaching map is null-homotopic because $Q_2$ is contractible, so that
	\[
	Q_2\cup CL\;\simeq\;Q_2\vee(CL/L)\;=\;Q_2\vee\Sigma L\;\simeq\;\Sigma L
	\]
	by~\cite[Prop.~0.18]{Hatcher(2002)}, the last step collapsing the contractible subcomplex $Q_2$. Therefore $K\simeq\Sigma L$.
	
	\emph{Step 2: the suspension of a graph.} Collapsing a spanning tree in each component is a homotopy equivalence~\cite[Prop.~0.17]{Hatcher(2002)} from $L$ onto a disjoint union $L'$ of $\kappa$ wedges of circles with $\ell$ circles in total, so $\Sigma L\simeq\Sigma L'$. Give $\Sigma L'$ its product cell structure: two $0$-cells (the suspension points) and one $(k+1)$-cell for each $k$-cell of $L'$. Its $1$-skeleton $\Gamma$ is a graph with two vertices and $\kappa$ edges, so $\Gamma\simeq\bigvee_{\kappa-1}S^1$; and each of the $\ell$ two-cells is attached along a loop of the form $\gamma\bar\gamma$ with $\gamma$ one of these edges, which is null-homotopic in $\Gamma$. Attaching $2$-cells along null-homotopic maps is, up to homotopy, the same as wedging on $2$-spheres~\cite[Prop.~0.18]{Hatcher(2002)}, so $\Sigma L\simeq\Gamma\vee\bigvee_{\ell}S^2\simeq(\bigvee_{\kappa-1}S^1)\vee(\bigvee_{\ell}S^2)$.
\end{proof}

\begin{lemma}[Attaching cells of dimension $\le1$]\label{lem:attach}
	Let $W$ be a CW complex with finitely many connected components $W_1,\dots,W_r$, each homotopy equivalent to a wedge of circles and $2$-spheres, and let $K$ be a connected CW complex obtained from $W$ by adjoining finitely many cells of dimension at most $1$. Then
	\[
	K\;\simeq\;\Big(\bigvee_{j=1}^{r}W_j\Big)\vee\Big(\bigvee_{t}S^1\Big)
	\]
	for some $t\ge0$. In particular $K$ is homotopy equivalent to a wedge of circles and $2$-spheres, $\pi_1(K)$ is free, and $H_2(K)\cong\bigoplus_{j}H_2(W_j)$.
\end{lemma}

\begin{proof}
	Adjoining the $0$-cells first gives $W^{+}=W\sqcup\{\text{adjoined }0\text{-cells}\}$, every component of which is homotopy equivalent to a wedge of circles and $2$-spheres (a point being the empty such wedge). The remaining cells are $1$-cells $e_1,\dots,e_N$, and we show by induction on $k$ that every component of $Z_k:=W^{+}\cup e_1\cup\dots\cup e_k$ has this form. Suppose this holds for $Z_{k-1}$ and attach $e_k$ along its endpoints $u,v$. If $u,v$ lie in the same component $Y$, then the attaching map $S^0\to Y$ is homotopic to a constant map, so $Y\cup e_k\simeq Y\vee S^1$~\cite[Prop.~0.18]{Hatcher(2002)}; if they lie in distinct components $Y_1,Y_2$, then collapsing the contractible arc $\bar e_k$ gives $Y_1\cup e_k\cup Y_2\simeq Y_1\vee Y_2$~\cite[Prop.~0.17]{Hatcher(2002)}. In either case the result is again a wedge of circles and $2$-spheres, which completes the induction; as $K=Z_N$ is connected, it is a single such wedge.
	
	No cell of dimension $\ge2$ was attached, so $H_2(K)\cong H_2(W^{+})\cong\bigoplus_jH_2(W_j)$ and $H_1(K)$ is free. Set $t=\operatorname{rank}H_1(K)-\sum_j\operatorname{rank}H_1(W_j)$, which is $\ge0$ because attaching $0$- and $1$-cells never lowers the rank of $H_1$. A wedge of circles and $2$-spheres is determined up to homotopy by the ranks of $H_1$ and $H_2$, so comparing these ranks gives the stated decomposition.
\end{proof}

\begin{lemma}[Complete tripartite cores]\label{lem:tripartite}
	Let $T(p,m,q)$ be the height-$2$ poset on disjoint level sets of sizes $p$ (maximal), $m$ (intermediate) and $q$ (minimal), in which every minimal point is below every intermediate point and every intermediate point is below every maximal point (so that no minimal point is covered by a maximal point). Then $T(p,m,q)$ is a core if and only if $p,m,q\ge 2$, and in that case
	\[
	|\mathcal{K}(T(p,m,q))|\;\simeq\;\bigvee_{(p-1)(m-1)(q-1)} S^{2}.
	\]
\end{lemma}

\begin{proof}
	Each intermediate point covers all $q$ minimal points and is covered by all $p$ maximal points, each minimal point is covered exactly by the $m$ intermediate points, and each maximal point covers exactly the $m$ intermediate points; so by Lemma~\ref{lem:crit} the poset is a core if and only if $p,m,q\ge2$. Assume this.
	
	A subset of $T(p,m,q)$ is a chain if and only if it meets each level in at most one point. Hence $K:=\mathcal{K}(T(p,m,q))$ is the join $D_p*G$ of the discrete set $D_p=\{x_1,\dots,x_p\}=\mathrm{mxl}(T(p,m,q))$ with the order complex $G$ of the two lower levels~\cite[Rem.~2.7.2]{Barmak(2011)}; here $G$ is the complete bipartite graph $K_{m,q}$, which is connected with $\operatorname{rank}H_1(G)=mq-m-q+1=(m-1)(q-1)$.
	
	Put $K_j=\bigcup_{i\le j}x_i*G$. For $1\le j<p$ we have $K_{j+1}=K_j\cup(x_{j+1}*G)$, the two pieces meet exactly in $G$, and $x_{j+1}*G$ is a cone on $G$; thus $K_{j+1}$ is the mapping cone of the inclusion $G\hookrightarrow K_j$. That inclusion is null-homotopic, since it factors through the cone $x_1*G$, so
	\[
	K_{j+1}\;\simeq\;K_j\vee\big((x_{j+1}*G)/G\big)\;=\;K_j\vee\Sigma G
	\]
	by~\cite[Prop.~0.18]{Hatcher(2002)}. Since $K_1$ is a cone, hence contractible, induction on $j$ gives $K\simeq\bigvee_{p-1}\Sigma G$. Finally $\Sigma G\simeq\bigvee_{(m-1)(q-1)}S^2$ by Step~2 of the proof of Lemma~\ref{lem:twocones}, so $K\simeq\bigvee_{(p-1)(m-1)(q-1)}S^2$.
\end{proof}

For the eight-point classification we repeatedly meet cores with exactly two maximal points or exactly two intermediate points. The following lemmas dispose of all such cores; by order-duality they apply equally to cores with exactly two minimal points.

\begin{lemma}[Free fundamental group]\label{lem:pifree}
	Let $X$ be a connected finite $T_0$-space of height $2$. If $X$ has at most two maximal points, or at most two minimal points, or exactly two intermediate points, then $\mathcal{K}(X)\simeq(\bigvee_{\beta_1}S^1)\vee(\bigvee_{\beta_2}S^2)$ with $\beta_i=\operatorname{rank}H_i(\mathcal{K}(X))$; in particular $\pi_1(\mathcal{K}(X))$ is free and $X\whe(\bigvee_{\beta_1}S^1)\vee(\bigvee_{\beta_2}S^2)$.
\end{lemma}

\begin{proof}
	Let $K=\mathcal{K}(X)$. If $X$ has only one maximal point, then $X$ has a maximum and is contractible, and the case of at most two minimal points is dual. Suppose first that $X$ has exactly two maximal points $a_1,a_2$. Every simplex of $K$ lies in a maximal chain ending at $a_1$ or at $a_2$, so $K=\overline{\mathrm{st}}(a_1)\cup\overline{\mathrm{st}}(a_2)$ is a union of two cones. Their intersection $L=\mathcal{K}(\hat U_{a_1}\cap\hat U_{a_2})$ is the order complex of a subposet of $X\setminus\mathrm{mxl}(X)$, which has height at most $1$; hence $L$ is a graph. By Lemma~\ref{lem:twocones}, with $\kappa$ the number of components of $L$ and $\ell=\operatorname{rank}H_1(L)$,
	\[
	K\simeq\Big(\bigvee_{\kappa-1}S^1\Big)\vee\Big(\bigvee_{\ell}S^2\Big).
	\]
	In particular $\pi_1(K)$ is free, $\beta_1=\kappa-1$ and $\beta_2=\ell$, so $K\simeq(\bigvee_{\beta_1}S^1)\vee(\bigvee_{\beta_2}S^2)$.
	
	Now suppose $X$ has exactly two intermediate points $b_1,b_2$; they are incomparable, since $\mathrm{h}(X)=2$. Every $2$-simplex of $K$ is a chain through $b_1$ or $b_2$, hence lies in $W=\overline{\mathrm{st}}(b_1)\cup\overline{\mathrm{st}}(b_2)$, a union of two contractible cones. Their intersection $W'=\overline{\mathrm{st}}(b_1)\cap\overline{\mathrm{st}}(b_2)$ is the order complex of the set of points comparable with both $b_1$ and $b_2$. These are the minimal points below both and the maximal points above both, and each such minimal point lies below each such maximal point (through $b_1$); so $W'$ is a complete bipartite graph, possibly with one side empty. If $W'\neq\varnothing$, then $W$ is connected and Lemma~\ref{lem:twocones} gives $W\simeq(\bigvee_{\kappa'-1}S^1)\vee(\bigvee_{\ell'}S^2)$ for suitable $\kappa'$ and $\ell'$. If $W'=\varnothing$, then $W=\overline{\mathrm{st}}(b_1)\sqcup\overline{\mathrm{st}}(b_2)$ is a disjoint union of two contractible complexes. Either way each component of $W$ is homotopy equivalent to a wedge of circles and $2$-spheres. The complex $K$ is obtained from $W$ by adjoining the remaining simplices, all of dimension $\le1$: the minimal points below no intermediate point, the maximal points above no intermediate point, and the edges $\{c,a\}$ given by direct covers $c\prec a$. Since $K$ is connected, Lemma~\ref{lem:attach} gives
	\[
	K\simeq\Big(\bigvee_{\beta_1}S^1\Big)\vee\Big(\bigvee_{\beta_2}S^2\Big),\qquad \beta_i=\operatorname{rank}H_i(K).
	\]
	In particular $\pi_1(K)$ is free.
\end{proof}

\begin{lemma}[Two-maximal-point Betti numbers]\label{lem:twomax}
	Let $X$ be a connected height-$2$ core with exactly two maximal points, $m$ intermediate points and $q$ minimal points. By Lemma~\ref{lem:co}, every intermediate point lies below both maximal points. Call a minimal point \emph{inner} if it lies below some intermediate point, and \emph{outer} otherwise. Let $s$ be the number of inner minimal points, and let $G$ be the bipartite graph on the inner minimal points and the intermediate points with an edge for each cover relation; let $E$ be its number of edges and $\kappa$ its number of connected components. Then every outer minimal point is covered by both maximal points, $G$ has minimal degree at least $2$ on both sides, and
	\[
	\beta_1(\mathcal{K}(X))=q-s-1+\kappa,\qquad \beta_2(\mathcal{K}(X))=E-s-m+\kappa .
	\]
	Hence $X\whe(\bigvee_{q-s-1+\kappa}S^1)\vee(\bigvee_{E-s-m+\kappa}S^2)$.
\end{lemma}

\begin{proof}
	An inner minimal point lies below some intermediate point $b$, hence below both maximal points through $b$, so it is covered by no maximal point. By Lemma~\ref{lem:crit}(2) it has at least two intermediate covers, that is, degree $\ge 2$ in $G$. An outer minimal point lies below no intermediate point, so all its immediate successors are maximal; by Lemma~\ref{lem:crit}(2) it is covered by at least two maximal points, that is, by both. Each intermediate point covers at least two minimal points (Lemma~\ref{lem:crit}(1)), all of them inner, so $G$ has minimal degree $\ge2$ on both sides.

	In particular every point of $X\setminus\mathrm{mxl}(X)$ lies below both maximal points $a_1,a_2$. As in the proof of Lemma~\ref{lem:pifree}, $\mathcal{K}(X)=\overline{\mathrm{st}}(a_1)\cup\overline{\mathrm{st}}(a_2)$, and the two cones intersect in $L=\mathcal{K}(\hat U_{a_1}\cap\hat U_{a_2})=\mathcal{K}(X\setminus\mathrm{mxl}(X))$. This is the graph $G$ together with the $q-s$ outer minimal points as isolated vertices, so $L$ has $\kappa+q-s$ connected components and $\operatorname{rank}H_1(L)=\operatorname{rank}H_1(G)=E-(s+m)+\kappa$. Lemma~\ref{lem:twocones} now gives
	\[
	\mathcal{K}(X)\simeq\Big(\bigvee_{q-s-1+\kappa}S^1\Big)\vee\Big(\bigvee_{E-s-m+\kappa}S^2\Big),
	\]
	which proves both formulas and the last assertion.
\end{proof}

A parallel analysis gives an equally explicit formula for the remaining case, that of a core with exactly two \emph{intermediate} points $b_1,b_2$. For such a core write $C_i=\hat U_{b_i}\cap\mathrm{mnl}(X)$ and $A_i=\hat F_{b_i}\cap\mathrm{mxl}(X)$ for $i=1,2$, and let $\delta$ be the number of direct covers.

\begin{lemma}[Two-intermediate-point Betti numbers]\label{lem:twomid}
	Let $X$ be a connected height-$2$ core with exactly two intermediate points $b_1,b_2$, with $p$ maximal and $q$ minimal points, and suppose that $C_1\cap C_2\neq\varnothing$ and $A_1\cap A_2\neq\varnothing$. (This is automatic when $p,q\le 3$: since $|C_i|,|A_i|\ge 2$ by Lemma~\ref{lem:co}, we get $|C_1\cap C_2|\ge|C_1|+|C_2|-q\ge1$, and likewise for the $A_i$.) Then
	\[
	\begin{aligned}
		\beta_2(\mathcal{K}(X))
		&= (|C_1\cap C_2|-1)(|A_1\cap A_2|-1),\\
		\beta_1(\mathcal{K}(X))
		&= |C_1\cup C_2|
		+ |A_1\cup A_2|
		+ \delta-p-q,
	\end{aligned}
	\]
	and
	\[
	X\whe
	\Big(\bigvee_{\beta_1}S^1\Big)
	\vee
	\Big(\bigvee_{\beta_2}S^2\Big).
	\]
\end{lemma}

\begin{proof}
	As in the proof of Lemma~\ref{lem:pifree}, let $W=\overline{\mathrm{st}}(b_1)\cup\overline{\mathrm{st}}(b_2)$; it contains every $2$-simplex of $\mathcal{K}(X)$, and $\mathcal{K}(X)$ is obtained from $W$ by adjoining cells of dimension at most $1$. The intersection $W'=\overline{\mathrm{st}}(b_1)\cap\overline{\mathrm{st}}(b_2)$ is the complete bipartite graph with parts $C_1\cap C_2$ and $A_1\cap A_2$, both non-empty by hypothesis. It is connected, and $\operatorname{rank}H_1(W')=(|C_1\cap C_2|-1)(|A_1\cap A_2|-1)$. By Lemma~\ref{lem:twocones}, $W\simeq\bigvee_{(|C_1\cap C_2|-1)(|A_1\cap A_2|-1)}S^2$, and by Lemma~\ref{lem:attach}, $H_2(\mathcal{K}(X))\cong H_2(W)$. This gives the formula for $\beta_2$.
	
	For $\beta_1$, count cells. The complex $\mathcal{K}(X)$ has $p+q+2$ vertices and $|C_1||A_1|+|C_2||A_2|$ triangles. Its edges are: $|C_1|+|C_2|$ edges from minimal to intermediate points; $|A_1|+|A_2|$ edges from intermediate to maximal points; and $|R|+\delta$ edges from minimal to maximal points, where
	\[
	R=\{(c,a):\ c\in C_i\text{ and }a\in A_i\text{ for some }i\}=(C_1\times A_1)\cup(C_2\times A_2)
	\]
	has $|R|=|C_1||A_1|+|C_2||A_2|-|C_1\cap C_2|\,|A_1\cap A_2|$. Hence
	\[
	\chi(\mathcal{K}(X))=(p+q+2)-\big(|C_1|+|C_2|\big)-\big(|A_1|+|A_2|\big)+|C_1\cap C_2|\,|A_1\cap A_2|-\delta.
	\]
	Since $\beta_0=1$, and using $|C_1|+|C_2|=|C_1\cap C_2|+|C_1\cup C_2|$ and the same identity for the $A_i$, we obtain
	\[
	\beta_1=1-\chi(\mathcal{K}(X))+\beta_2=|C_1\cup C_2|+|A_1\cup A_2|+\delta-p-q.
	\]
	The weak homotopy type follows from Lemma~\ref{lem:pifree} (exactly two intermediate points).
\end{proof}

We close the preliminaries with an elementary reconstruction principle, used in the eight-point classification to recast the case analysis as a classification of symmetry orbits of combinatorial data.

\begin{lemma}[Determination principle]\label{lem:determination}
	Let $X$ be a finite $T_0$-space of height $2$ with level sets $\mathrm{mxl}(X)$, $B_X=\{b_1,\dots,b_m\}$ and $\mathrm{mnl}(X)$. For each $i$ set $C_i=\hat U_{b_i}\cap\mathrm{mnl}(X)$ and $A_i=\hat F_{b_i}\cap\mathrm{mxl}(X)$, and let $\Delta\subseteq\mathrm{mnl}(X)\times\mathrm{mxl}(X)$ be the set of direct covers $c\prec a$. Then the order of $X$ is recovered from the datum $(C_1,\dots,C_m,A_1,\dots,A_m,\Delta)$ by
	\[
	\begin{aligned}
		c<b_i &\iff c\in C_i, \qquad
		b_i<a \iff a\in A_i,\\
		c<a &\iff
		\big(c\in C_i\text{ and }a\in A_i\text{ for some }i\big)
		\ \text{or}\ (c,a)\in\Delta,
	\end{aligned}
	\]
	and points of the same level are incomparable. Conversely, for any datum these relations define a partial order of height at most $2$. Its level sets are the given ones provided each $C_i$ and each $A_i$ is non-empty and every designated minimal (resp.\ maximal) point lies below (resp.\ above) at least one point; the second condition holds in particular whenever the resulting space is connected and has more than one point. Its set of direct covers is $\Delta$ provided $\Delta\cap\bigcup_i(C_i\times A_i)=\varnothing$. Two such spaces are homeomorphic if and only if one datum is carried onto the other by a level-preserving bijection, that is, by an element of $\mathrm{Sym}(\mathrm{mxl}(X))\times\mathrm{Sym}(B_X)\times\mathrm{Sym}(\mathrm{mnl}(X))$.
\end{lemma}

\begin{proof}
	The displayed equivalences specify every comparability between points of different levels, and points of the same level are incomparable; this is exactly the order of $X$. Conversely, for any datum the displayed relation is irreflexive and transitive (a composite $c<b_i<a$ is included by construction), so it is a strict partial order in which every chain has at most three points. Its minimal, intermediate and maximal points are the designated ones under the stated non-emptiness conditions. A pair $(c,a)\in\Delta$ is a cover unless $c<b_i<a$ for some $i$, that is, unless $(c,a)\in\bigcup_i(C_i\times A_i)$, which proves the statement about direct covers. Finally, a homeomorphism of finite $T_0$-spaces is an order-isomorphism. The three levels are defined order-theoretically (minimal points, maximal points and the rest), so every order-isomorphism preserves them and restricts to bijections of the level sets; conversely, every level-preserving bijection that carries one datum onto the other is an order-isomorphism, hence a homeomorphism.
\end{proof}

\section{Main results}\label{sec:main}

We first explain how the small models were found (Section~\ref{subsec:cw}). We then classify the connected cores of height $2$ with seven and with eight points (Sections~\ref{subsec:seven} and~\ref{subsubsec:eight}), which proves Theorem~\ref{thm:A}, and finally determine all spaces whose minimal finite models have at most eight points (Section~\ref{subsec:small}).

\subsection{Small models from regular CW complexes}\label{subsec:cw}

Finite models arise naturally from regular CW complexes, that is, CW complexes whose characteristic maps are homeomorphisms onto their images.

\begin{remark}\label{rem:facepos}
	Let $K$ be a regular CW complex, and let $\mathcal{X}(K)$ be its face poset: the set of its cells, ordered by the face relation ($e\le e'$ if and only if $e\subseteq\overline{e'}$). The order complex $\mathcal{K}(\mathcal{X}(K))$ is the barycentric subdivision $K'$ of $K$, and $|K'|$ is homeomorphic to $K$~\cite[Sec.~7.1]{Barmak(2011)}. Hence, by Theorem~\ref{thm:mccord}(2), $\mathcal{X}(K)$ is a finite model of $K$ whose points are the cells of $K$; see also~\cite{Barmak-Minian(2008)}.
\end{remark}

The regular CW structure on $S^k$ with two cells in each dimension $0,1,\dots,k$ has face poset $\mathbb{S}^{k}(S^0)$. Wedging such structures at a common $0$-cell gives a regular CW structure on $S^{k_1}\vee\cdots\vee S^{k_r}$ with $\sum_{i=1}^{r}(2k_i+2)-(r-1)$ cells; for $S^2\vee S^1$, $S^2\vee S^2$, $S^1\vee S^1\vee S^2$, $S^1\vee S^2\vee S^2$ and $S^2\vee S^2\vee S^2$ this already yields finite models with $9$, $11$, $12$, $14$ and $16$ points. Much smaller ones are found by the following heuristic: let different parts of the complex, in particular different wedge summands, share lower-dimensional cells. The result is a regular CW complex $K_1$ with fewer cells, whose underlying space is in general \emph{not} homeomorphic to the original one; by Remark~\ref{rem:facepos} its face poset $\mathcal{X}(K_1)$ is a finite model of $|K_1|$.

Figure~\ref{b} carries this out for $S^2\vee S^1$: letting the circle share both $0$-cells and one $1$-cell with the equator of the $2$-sphere leaves two $0$-cells, three $1$-cells and two $2$-cells, and the associated seven-point finite space is the model of Fig.~\ref{d}(a$^\ast$). The same device produces the remaining candidates. Three $2$-cells attached along a common circle give a seven-point model of $S^2\vee S^2$ (Fig.~\ref{h}(a$^\ast$)); a $2$-sphere together with two arcs joining the two $0$-cells of its equator gives an eight-point model of $S^1\vee S^1\vee S^2$ (Fig.~\ref{k}(a$^\ast$)); three $2$-cells on a common circle together with one such arc give one of $S^1\vee S^2\vee S^2$ (Fig.~\ref{lprime}(b$^\ast$)); and four $2$-cells on a common circle give one of $S^2\vee S^2\vee S^2$ (Fig.~\ref{hprime}(a$^\ast$)). By Theorems~\ref{thm:seven} and~\ref{thm:eight} these are minimal finite models, and so are their opposites.

The heuristic only produces candidates: in every case the weak homotopy type of the associated finite space is verified independently by the finite-space methods of Section~\ref{sec:prelim}, and minimality is proved in Theorems~\ref{thm:seven} and~\ref{thm:eight}. For the five wedges above, the modification reduces the number of cells from $9$, $11$, $12$, $14$ and $16$ to $7$, $7$, $8$, $8$ and $8$, respectively. These numbers are optimal: if a regular CW complex is weak homotopy equivalent to one of the five wedges, then its face poset is a finite model of that wedge (Remark~\ref{rem:facepos}), so by Lemmas~\ref{lem:lb7}(2) and~\ref{lem:lb8} the complex has at least seven cells in the first two cases and at least eight in the other three.

\begin{figure}[htbp]
	\centering
	\begin{adjustbox}{max width=\linewidth}%
	\begin{adjustbox}{max width=\linewidth, max totalheight=0.26\textheight, valign=c}\begin{tikzpicture}
		\shade[ball color = pink!50] (1.2,0) circle (1.2cm);
		\draw (1.2,0) circle (1.2cm);
		\draw (0,0) arc (180:360:1.2 and 0.3);         %
		\draw[dashed] (2.4,0) arc (0:180:1.2 and 0.3); %

		\draw (3.4,0) circle (1cm);

		\fill (0,0) circle (1.5pt);
		\fill (2.4,0) circle (1.5pt);
		\fill (4.4,0) circle (1.5pt);

		\node[left] at (-.2,0) {$c_1$};
		\node at (1.7,0.5) {$b_1$};
		\node at (1.7,-0.5) {$b_2$};
		\node at (1.2,0.6) {$a_1$};
		\node at (1.2,-0.6) {$a_2$};
		\node[right] at (2.45,0) {$c_2$};
		\node[right] at (4.5,0) {$c_3$};
		\node[above] at (3.4,1) {$a_3$};
		\node[below] at (3.4,-1) {$a_4$};
	\end{tikzpicture}\end{adjustbox}%
	\hspace{0.9cm}%
	\begin{adjustbox}{max width=\linewidth, max totalheight=0.26\textheight, valign=c}\begin{tikzpicture}
		\shade[ball color = pink!50] (1.2,0) circle (1.2cm);
		\draw (1.2,0) circle (1.2cm);
		\draw (0,0) arc (180:360:1.2 and 0.3);         %
		\draw[dashed] (2.4,0) arc (0:180:1.2 and 0.3); %

		\draw (2.4,0) arc (0:180:1.2 and 1.9);

		\fill (0,0) circle (1.5pt);
		\fill (2.4,0) circle (1.5pt);

		\node[left] at (-.2,0) {$c_1$};
		\node at (1.7,0.5) {$b_1$};
		\node at (1.7,-0.5) {$b_2$};
		\node at (1.2,0.6) {$a_1$};
		\node at (1.2,-0.6) {$a_2$};
		\node[right] at (2.6,0) {$c_2$};
		\node[above] at (1.2,1.9) {$a_3$};
	\end{tikzpicture}\end{adjustbox}%
	\hspace{0.9cm}%
	\begin{adjustbox}{max width=\linewidth, max totalheight=0.26\textheight, valign=c}\begin{tikzpicture}
		[acteur/.style={circle, fill=black,thick, inner sep=2pt, minimum size=0.2cm}]
		\node (1) at ( 4.5,0) [acteur,label=below:$c_1$]{};
		\node (2) at ( 5.5,0) [acteur,label=below:$c_2$]{};
		\node (3) at ( 4,1) [acteur,label=left:$b_1$]{};
		\node (4) at (5,1) [acteur][label=left:$b_2$]{};
		\node (5) at ( 6,2) [acteur,label=above:$a_3$]{};
		\node (6) at (4,2) [acteur][label=above:$a_1$]{};
		\node (7) at (5,2) [acteur][label=above:$a_2$]{};
		\draw [-, thick, red] (1) -- (3);
		\draw [-, thick, red] (1) -- (4);
		\draw [-, thick, red] (1) -- (5);
		\draw [-, thick, red] (2) -- (3);
		\draw [-, thick, red] (2) -- (4);
		\draw [-, thick, red] (2) -- (5);
		\draw [-, thick, red] (3) -- (6);
		\draw [-, thick, red] (3) -- (7);
		\draw [-, thick, red] (4) -- (6);
		\draw [-, thick, red] (4) -- (7);
	\end{tikzpicture}\end{adjustbox}%
	\end{adjustbox}
	\caption{The regular CW structure on $S^2\vee S^1$ obtained by wedging the summands, with nine cells (left); the modified structure obtained from it by identifying $c_3$ with $c_1$ and $a_4$ with $b_1$, with seven cells (middle); and its associated finite space (right), which is the model of Fig.~\ref{d}(a$^\ast$). The letters $a$, $b$, $c$ label maximal, intermediate and minimal points of the poset, not the dimensions of the cells: here the $1$-cell $a_3$ is a maximal point}
	\label{b}
\end{figure}

\subsection{Seven-point classification}\label{subsec:seven}

\begin{lemma}\label{lem:lb7}
	\begin{enumerate}
		\item A connected core with at most six points is weak homotopy equivalent to $\bigvee_k S^1$ for some $0\le k\le4$ (a point if $k=0$) or to $S^2$. The last case occurs only for $\mathbb{S}^2(S^0)$, which has six points.
		\item Every finite model of $S^{2}\vee S^{1}$ or of $S^{2}\vee S^{2}$ has at least seven points. Consequently every seven-point finite model of either space is a minimal finite model, and every regular CW structure on either space has at least seven cells.
	\end{enumerate}
\end{lemma}

\begin{proof}
	(1) Let $X$ be a connected core with $\#X\le6$, and let $h=\mathrm{h}(X)$. If $X$ is a point there is nothing to prove, so let $X\neq\ast$. By Theorem~\ref{thm:barmak7}, $2h+2\le\#X\le6$, so $h\le2$. If $h=2$, then $\#X=6=2h+2$ and $X\cong\mathbb{S}^2(S^0)$ by the same theorem, and $\mathbb{S}^2(S^0)$ is a finite model of $S^2$. If $h=1$, then $\mathcal{K}(X)$ is the Hasse diagram of $X$: a connected bipartite graph on the $i$ maximal and $j$ minimal points of $X$, with $i+j\le6$ and at most $ij$ edges. Hence $X\whe\bigvee_k S^1$ with
	\[
	k=\#E(\mathcal{H}(X))-(i+j)+1\le ij-i-j+1=(i-1)(j-1)\le4,
	\]
	the last inequality because $i+j\le6$.
	
	(2) Let $X$ be a finite model of $S^2\vee S^1$ or $S^2\vee S^2$ with the least possible number of points. Then $X$ is a minimal finite space~\cite[Rem.~3.1.2]{Barmak(2011)}, and it is connected because the target space is. If $\#X\le6$, then by (1) the space $X$ is weak homotopy equivalent to a wedge of circles or to $S^2$. This is impossible by Remark~\ref{rem:wedgeinv}: a wedge of circles has $H_2=0$, while $S^2$ has $H_1=0$ and $H_2\cong\mathbb{Z}$, whereas $H_1(S^2\vee S^1)\cong\mathbb{Z}$ and $H_2(S^2\vee S^2)\cong\mathbb{Z}^2$. Hence every finite model of either space has at least seven points. The last assertion follows from Remark~\ref{rem:facepos}: the face poset of a regular CW structure is a finite model whose points are the cells.
\end{proof}

\begin{theorem}\label{thm:seven}
	Every connected seven-point core of height $2$ is weak homotopy equivalent to a wedge of circles, to $S^{2}\vee S^{1}$, or to $S^{2}\vee S^{2}$. Consequently, the minimal finite models of $S^{2}\vee S^{1}$ and of $S^{2}\vee S^{2}$ have exactly seven points, and up to homeomorphism they are the following:
	\begin{enumerate}
		\item the minimal finite models of $S^{2}\vee S^{1}$ are precisely the two posets shown in Fig.~\ref{d};
		\item the minimal finite models of $S^{2}\vee S^{2}$ are precisely the three posets shown in Fig.~\ref{h}.
	\end{enumerate}
\end{theorem}

\begin{figure}[htbp]
	\centering
			\begin{adjustbox}{max width=\linewidth, max totalheight=0.60\textheight}\begin{tikzpicture}
				[acteur/.style={circle, fill=black,thick, inner sep=2pt, minimum size=0.2cm}]
				\begin{scope}[shift={(0,0)}]
					\node (1) at (0,0) [acteur,label=below:$c_1$]{};
					\node (2) at (1,0) [acteur,label=below:$c_2$]{};
					\node (3) at (2,0) [acteur,label=below:$c_3$]{};
					\node (4) at (0,1) [acteur,label=left:$b_1$]{};
					\node (5) at (1,1) [acteur,label=left:$b_2$]{};
					\node (6) at (0.5,2) [acteur,label=above:$a_1$]{};
					\node (7) at (1.5,2) [acteur,label=above:$a_2$]{};
					\draw[-, thick, red] (1) -- (4);
					\draw[-, thick, red] (1) -- (5);
					\draw[-, thick, red] (2) -- (4);
					\draw[-, thick, red] (2) -- (5);
					\draw[-, thick, red] (3) -- (6);
					\draw[-, thick, red] (3) -- (7);
					\draw[-, thick, red] (4) -- (6);
					\draw[-, thick, red] (4) -- (7);
					\draw[-, thick, red] (5) -- (6);
					\draw[-, thick, red] (5) -- (7);
					\node at (1,-0.8) {(a)};
				\end{scope}
				\begin{scope}[shift={(6,0)}]
					\node (1) at ( 0.5,0) [acteur,label=below:$c_1$]{};
					\node (2) at ( 1.5,0) [acteur,label=below:$c_2$]{};
					\node (3) at ( 0,1) [acteur,label=left:$b_1$]{};
					\node (4) at (1,1) [acteur][label=left:$b_2$]{};
					\node (5) at (2,2) [acteur,label=above:$a_3$]{};
					\node (6) at (0,2) [acteur][label=above:$a_1$]{};
					\node (7) at (1,2) [acteur][label=above:$a_2$]{};
					\draw [-, thick, red] (1) -- (3);
					\draw [-, thick, red] (1) -- (4);
					\draw [-, thick, red] (1) -- (5);
					\draw [-, thick, red] (2) -- (3);
					\draw [-, thick, red] (2) -- (4);
					\draw [-, thick, red] (2) -- (5);
					\draw [-, thick, red] (3) -- (6);
					\draw [-, thick, red] (3) -- (7);
					\draw [-, thick, red] (4) -- (6);
					\draw [-, thick, red] (4) -- (7);
					\node at (1,-0.8) {(a$^\ast$) Dual of (a)};
				\end{scope}
			\end{tikzpicture}\end{adjustbox}
			\caption{Minimal finite models of $S^{2}\vee S^{1}$}
			\label{d}
		\end{figure}

\begin{figure}[htbp]
	\centering
			\begin{adjustbox}{max width=\linewidth, max totalheight=0.60\textheight}\begin{tikzpicture}
				[acteur/.style={circle, fill=black,thick, inner sep=2pt, minimum size=0.2cm}]
				
				\begin{scope}[shift={(0,0)}]
					\node (1d) at (0,2) [acteur,label=above:$a_1$]{};
					\node (2d) at (1,2) [acteur,label=above:$a_2$]{};
					\node (3d) at (0,1) [acteur,label=left:$b_1$]{};
					\node (4d) at (1,1) [acteur,label=left:$b_2$]{};
					\node (5d) at (0,0) [acteur,label=below:$c_1$]{};
					\node (6d) at (1,0) [acteur,label=below:$c_2$]{};
					\node (7d) at (2,0) [acteur,label=below:$c_3$]{};
					\draw [-, thick, red] (1d)--(3d);
					\draw [-, thick, red] (1d)--(4d);
					\draw [-, thick, red] (2d)--(3d);
					\draw [-, thick, red] (2d)--(4d);
					\draw [-, thick, red] (3d)--(5d);
					\draw [-, thick, red] (3d)--(6d);
					\draw [-, thick, red] (3d)--(7d);
					\draw [-, thick, red] (4d)--(5d);
					\draw [-, thick, red] (4d)--(6d);
					\draw [-, thick, red] (4d)--(7d);
					\node at (1,-0.8) {(a)};
				\end{scope}
				
				\begin{scope}[shift={(5,0)}]
					\node (1) at ( 0,0) [acteur,label=below:$c_1$]{};
					\node (2) at ( 1,0) [acteur,label=below:$c_2$]{};
					\node (3) at (0,1) [acteur,label=left:$b_1$]{};
					\node (4) at (1,1) [acteur,label=left:$b_2$]{};
					\node (5) at (0,2) [acteur,label=above:$a_1$]{};
					\node (6) at (1,2) [acteur,label=above:$a_2$]{};
					\node (7) at (2,2) [acteur,label=above:$a_3$]{};
					\draw [-, thick, red] (1)--(3);
					\draw [-, thick, red] (1)--(4);
					\draw [-, thick, red] (2)--(3);
					\draw [-, thick, red] (2)--(4);
					\draw [-, thick, red] (3)--(5);
					\draw [-, thick, red] (3)--(6);
					\draw [-, thick, red] (3)--(7);
					\draw [-, thick, red] (4)--(5);
					\draw [-, thick, red] (4)--(6);
					\draw [-, thick, red] (4)--(7);
					\node at (1,-0.8) {(a$^\ast$) Dual of (a)};
				\end{scope}
				
				\begin{scope}[shift={(10,0)}]
					\node (1b) at (0,0) [acteur,label=below:$c_1$]{};
					\node (2b) at (1,0) [acteur,label=below:$c_2$]{};
					\node (4b) at (0,1) [acteur,label=left:$b_1$]{};
					\node (5b) at (1,1) [acteur,label=left:$b_2$]{};
					\node (3b) at (2,1) [acteur,label=below:$b_3$]{};
					\node (6b) at (0,2) [acteur,label=above:$a_1$]{};
					\node (7b) at (1,2) [acteur,label=above:$a_2$]{};
					\draw [-, thick, red] (1b)--(3b);
					\draw [-, thick, red] (1b)--(4b);
					\draw [-, thick, red] (1b)--(5b);
					\draw [-, thick, red] (2b)--(3b);
					\draw [-, thick, red] (2b)--(4b);
					\draw [-, thick, red] (2b)--(5b);
					\draw [-, thick, red] (3b)--(6b);
					\draw [-, thick, red] (3b)--(7b);
					\draw [-, thick, red] (4b)--(6b);
					\draw [-, thick, red] (4b)--(7b);
					\draw [-, thick, red] (5b)--(6b);
					\draw [-, thick, red] (5b)--(7b);
					\node at (1,-0.8) {(b) $=$ (b$^\ast$)};
				\end{scope}
			\end{tikzpicture}\end{adjustbox}
			\caption{Minimal finite models of $S^{2}\vee S^{2}$}
			\label{h}
		\end{figure}

\begin{proof}
	Let $X$ be a connected seven-point core of height $2$. Replacing $X$ by its opposite $X^{\mathrm{op}}$ if necessary (which is again such a core and, by Theorem~\ref{thm:mccord}(2), has the same weak homotopy type), we may assume
	\[
	\#\mathrm{mxl}(X)\leq \#\mathrm{mnl}(X).
	\]
	Since $X$ is a core with more than one point, it is not contractible, so $\#\mathrm{mxl}(X)\geq 2$ and $\#\mathrm{mnl}(X)\geq 2$ by Remark~\ref{rem:mxlmnl}. Since $X$ is connected with more than one point, $\mathrm{mxl}(X)\cap \mathrm{mnl}(X)=\varnothing$. As $\mathrm{h}(X)=2$, we have $B_X\neq \varnothing$, and the elements of $B_X$ form an antichain.
	
	Whether $X$ is a core is decided by Lemma~\ref{lem:crit}: every intermediate point must lie below at least two maximal points and above at least two minimal points, and every minimal (resp.\ maximal) point must have at least two immediate successors (resp.\ predecessors). In a height-$2$ poset, the immediate successors of a minimal point $c$ are the intermediate points above $c$ together with the maximal points covering $c$ directly, and dually for maximal points.
	
	Write $(p,m,q)=(\#\mathrm{mxl}(X),\#B_X,\#\mathrm{mnl}(X))$ for the level vector of $X$. Then $p,q\ge2$ by Remark~\ref{rem:mxlmnl}, $m\ge1$ because $\mathrm{h}(X)=2$, $p+m+q=7$, and $p\le q$ by the reduction above. These constraints have exactly four solutions: for $m=1$ we need $p+q=6$ with $2\le p\le q$, giving $(p,q)=(2,4)$ or $(3,3)$; for $m=2$, $p+q=5$ forces $(p,q)=(2,3)$; for $m=3$, $p+q=4$ forces $(p,q)=(2,2)$; and $m\ge4$ would give $p+q\le3$, contradicting $p,q\ge2$. Hence
	\[
	(p,m,q)\in\{(2,1,4),\ (3,1,3),\ (2,2,3),\ (2,3,2)\},
	\]
	the duals having level vectors $(4,1,2)$, $(3,1,3)$, $(3,2,2)$ and $(2,3,2)$; in particular $\#B_X\in\{1,2,3\}$. We treat these cases separately below, in each case determining the possible cores and computing their weak homotopy types with the lemmas listed in Remark~\ref{rem:method}. The outcome is summarised in the following table:
	\[
	\renewcommand{\arraystretch}{1.15}
	\begin{array}{c c l}
		(p,m,q),\ p\le q & \#B_X & \text{cores and which targets occur}\\ \hline
		(2,1,4) & 1 & \text{no core (Case 1); no target}\\
		(3,1,3) & 1 & S^1\vee S^1 \text{ and } S^1\vee S^1\vee S^1 \text{ (Case 1); no target}\\
		(2,2,3) & 2 & \text{one } S^2\vee S^1 \text{ and one } S^2\vee S^2 \text{ (Case 2)}\\
		(2,3,2) & 3 & \text{one } S^2\vee S^2 \text{ (Case 3)}\\ \hline
	\end{array}
	\]
	Thus the targets occur only for $\#B_X\ge2$. Adjoining duals, they occur exactly as follows: two models of $S^2\vee S^1$ (the $(2,2,3)$ core and its $(3,2,2)$ dual) and three of $S^2\vee S^2$ (the $(2,2,3)$ and dual $(3,2,2)$ tripartite cores, and the self-dual $(2,3,2)$ one). The three cases below establish each row.
	
	\medskip
	\noindent\textbf{Case 1: $\#B_X=1$.}
	
	Let $B_X=\{b\}$, so $(\#\mathrm{mxl}(X),\#\mathrm{mnl}(X))$ is $(2,4)$ or $(3,3)$. By Lemma~\ref{lem:onemid}, $X$ is weak homotopy equivalent to a wedge of circles; in particular it is a model of neither $S^2\vee S^1$ nor $S^2\vee S^2$, both of which have $H_2\neq 0$. For completeness we determine these cores. By Lemma~\ref{lem:co}, $b$ lies below at least two maximal points and above at least two minimal points. If $b$ lies below every maximal point, then a minimal point $c<b$ lies below every maximal point through $b$, so $b$ is its only immediate successor and $c$ is a beat point. Hence $b$ does not lie below every maximal point. In particular the split $(2,4)$ admits no core.
	
	For the split $(3,3)$, $b$ lies below exactly two maximal points, say $a_1,a_2$. A minimal point $c<b$ is then below $a_1$ and $a_2$ through $b$, so its second immediate successor must be $a_3$. The maximal point $a_1$ needs an immediate predecessor other than $b$, that is, a minimal point below $a_1$ but not below $b$; so at most two minimal points lie below $b$. By Lemma~\ref{lem:co}, exactly two minimal points, say $c_1,c_2$, lie below $b$, and both are covered by $a_3$. Now $c_3$ is the only minimal point not below $b$, so the extra immediate predecessor of $a_1$ is $c_3$, and likewise for $a_2$; that is, $c_3$ is covered by $a_1$ and by $a_2$. The only remaining freedom is whether $c_3<a_3$, and in both cases Lemma~\ref{lem:crit} is satisfied. This gives exactly two cores. Their order complexes have $7$ vertices, $12+\varepsilon$ edges and $4$ triangles, where $\varepsilon=1$ if $c_3<a_3$ and $\varepsilon=0$ otherwise. Since $H_2=0$, this gives $\beta_1=1-\chi(\mathcal{K}(X))=2+\varepsilon$, so these cores are weak homotopy equivalent to $S^1\vee S^1$ (for $\varepsilon=0$) and to $S^1\vee S^1\vee S^1$ (for $\varepsilon=1$).

	\medskip
	\noindent\textbf{Case 2: $\#B_X=2$.}
	
	Here $\#\mathrm{mxl}(X)=2$ and $\#\mathrm{mnl}(X)=3$. Write
	\[
	\mathrm{mxl}(X)=\{a_1,a_2\}, \quad B_X=\{b_1,b_2\}, \quad \mathrm{mnl}(X)=\{c_1,c_2,c_3\}.
	\]
	By Lemma~\ref{lem:co}, each $b_i$ lies below at least two maximal points; as there are only two, $\hat F_{b_i}\cap\mathrm{mxl}(X)=\{a_1,a_2\}$ for $i=1,2$. Let $C_i=\hat U_{b_i}\cap\mathrm{mnl}(X)$, so $|C_i|\in\{2,3\}$. A minimal point $c$ below exactly one intermediate point $b_i$ lies below both maximal points through $b_i$, so $b_i$ is its only immediate successor and $c$ is a beat point. Hence every minimal point lies below both intermediate points or below neither, that is, $C_1=C_2$. This leaves two configurations.
	\begin{enumerate}
		\item $C_1=C_2=\{c_1,c_2\}$, and $c_3$ lies below neither intermediate point. Then all immediate successors of $c_3$ are maximal, and by Lemma~\ref{lem:crit}(2) it needs at least two of them; as there are only two maximal points, $c_3$ is covered by both $a_1$ and $a_2$. This determines a unique poset, and it is a core by Lemma~\ref{lem:crit}. We apply Lemma~\ref{lem:twomax} with $m=2$ and $q=3$. The minimal points $c_1,c_2$ are inner, each below both intermediate points, so the graph $G$ is the complete bipartite graph $K_{2,2}$ with $E=4$ edges, $s=2$ and $\kappa=1$; and $c_3$ is outer. Hence $\beta_1=q-s-1+\kappa=1$ and $\beta_2=E-s-m+\kappa=1$, so $X\whe S^{2}\vee S^{1}$ (Fig.~\ref{d}(a); its dual is $(a^{\ast})$).
		\item $C_1=C_2=\{c_1,c_2,c_3\}$: each $b_i$ lies below both maximal points and above all three minimal points, giving the complete tripartite poset $T(2,2,3)$. By Lemma~\ref{lem:tripartite} it is a core and $X\whe\bigvee_{(2-1)(2-1)(3-1)}S^2=S^{2}\vee S^{2}$ (Fig.~\ref{h}(a); its dual is $(a^{\ast})$).
	\end{enumerate}
	Thus Case 2 yields one model of $S^{2}\vee S^{1}$ and one of $S^{2}\vee S^{2}$, each together with its dual.
	
	\medskip
	\noindent\textbf{Case 3: $\#B_X=3$.}
	
	Here $\#\mathrm{mxl}(X)=\#\mathrm{mnl}(X)=2$. Write
	\[
	\mathrm{mxl}(X)=\{a_1,a_2\}, \quad B_X=\{b_1,b_2,b_3\}, \quad \mathrm{mnl}(X)=\{c_1,c_2\}.
	\]
	By Lemma~\ref{lem:co}, each $b_i$ lies below both maximal points and above both minimal points (there being only two of each), so $\hat F_{b_i}\cap\mathrm{mxl}(X)=\{a_1,a_2\}$ and $\hat U_{b_i}\cap\mathrm{mnl}(X)=\{c_1,c_2\}$ for $i=1,2,3$. This determines $X$ uniquely up to homeomorphism: it is the complete tripartite poset $T(2,3,2)$. By Lemma~\ref{lem:tripartite} it is a core and $X\whe\bigvee_{(2-1)(3-1)(2-1)}S^2=S^{2}\vee S^{2}$ (Fig.~\ref{h}(b)).
	
	\medskip
	Collecting the cases: a connected seven-point core of height $2$ has $\#B_X\in\{1,2,3\}$, and it is weak homotopy equivalent to a wedge of circles (Case~1), to $S^2\vee S^1$ or $S^2\vee S^2$ (Case~2), or to $S^2\vee S^2$ (Case~3). This proves the first assertion.
	
	For the second, Lemma~\ref{lem:lb7}(2) shows that every finite model of $S^2\vee S^1$ or $S^2\vee S^2$ has at least seven points, and Case~2 exhibits seven-point models of both spaces; so their minimal finite models have exactly seven points. A minimal finite model is a minimal finite space~\cite[Rem.~3.1.2]{Barmak(2011)}; it is connected because its target is; and it has height $2$ by Remark~\ref{rem:ht2}. So the minimal finite models of the two spaces are exactly the connected seven-point cores of height $2$ with the corresponding weak homotopy types. Adjoining the order-duals, which by Theorem~\ref{thm:mccord}(2) have the same types, the case analysis shows that these are the two posets of Fig.~\ref{d} for $S^{2}\vee S^{1}$ and the three posets of Fig.~\ref{h} for $S^{2}\vee S^{2}$.
\end{proof}

\subsection{Eight-point classification}\label{subsubsec:eight}

\begin{lemma}\label{lem:lb8}
	None of the spaces $S^{1}\vee S^{1}\vee S^{2}$, $S^{1}\vee S^{1}\vee S^{1}\vee S^{2}$, $S^{1}\vee S^{2}\vee S^{2}$, $S^{2}\vee S^{2}\vee S^{2}$, $S^{2}\vee S^{2}\vee S^{2}\vee S^{2}$, or $S^{1}\vee S^{1}\vee S^{2}\vee S^{2}$ admits a finite model on fewer than eight points. Moreover, any eight-point minimal finite model of one of these spaces has height $2$.
\end{lemma}

\begin{proof}
	Let $Y$ be one of the six spaces and suppose it has a finite model on at most seven points. Among all such models choose one, $X$, of least cardinality. Then $X$ is a minimal finite space~\cite[Rem.~3.1.2]{Barmak(2011)} (removing beat points decreases the cardinality and preserves the weak homotopy type, by Theorem~\ref{thm:mccord}(1)), so $X$ is a core with $\#X\le 7$; and $X$ is connected since $Y$ is. Each $Y$ has $H_2(Y)\neq 0$, so $\dim\mathcal{K}(X)\ge 2$, that is, $\mathrm{h}(X)\ge 2$. If $\mathrm{h}(X)\ge 3$, then $\#X\ge 2\mathrm{h}(X)+2\ge 8$ by Theorem~\ref{thm:barmak7}, a contradiction; hence $\mathrm{h}(X)=2$.
	
	If $\#X\le 6$, then by Lemma~\ref{lem:lb7}(1), $X$ is weak homotopy equivalent to a point, to $\bigvee_k S^1$ with $k\le 4$, or to $S^2$, with Betti numbers $(\beta_1,\beta_2)$ equal to $(0,0)$, $(k,0)$ or $(0,1)$, respectively. If $\#X=7$, then $X$ is a connected seven-point core of height $2$, so by Theorem~\ref{thm:seven} it is weak homotopy equivalent to a wedge of circles, to $S^2\vee S^1$, or to $S^2\vee S^2$, with Betti numbers $(k,0)$, $(1,1)$ or $(0,2)$. In all cases the Betti numbers differ from those of $Y$ (Remark~\ref{rem:wedgeinv}), since
	\[
	(\beta_1,\beta_2)(Y)\in\{(2,1),(3,1),(1,2),(0,3),(0,4),(2,2)\}.
	\]
	Thus $X\not\whe Y$, a contradiction, and so every finite model of $Y$ has at least eight points.
	
	Finally, let $X$ be an eight-point minimal finite model of some $Y$; it is a core of height $\ge 2$. If $\mathrm{h}(X)\ge 3$, then $8=\#X\ge 2\mathrm{h}(X)+2\ge 8$, so $\mathrm{h}(X)=3$ and, by Theorem~\ref{thm:barmak7}, $X\cong\mathbb{S}^{3}(S^0)\whe S^3$. But $S^3\not\whe Y$ by Remark~\ref{rem:wedgeinv}. Hence $\mathrm{h}(X)=2$.
\end{proof}

\noindent By Lemma~\ref{lem:lb8}, it suffices to classify the connected eight-point cores of height $2$ and to identify the minimal finite models among them. In particular, we shall see that $S^{1}\vee S^{1}\vee S^{1}\vee S^{2}$, $S^{2}\vee S^{2}\vee S^{2}\vee S^{2}$ and $S^{1}\vee S^{1}\vee S^{2}\vee S^{2}$ also admit eight-point models.

\begin{remark}[Counting convention and duality]\label{rem:count}
	Two finite $T_0$-spaces are homeomorphic precisely when the corresponding posets are order-isomorphic; for height-$2$ posets this means a bijection preserving the three levels and the order (Lemma~\ref{lem:determination}). We count minimal finite models up to homeomorphism. The involution $X\mapsto X^{\mathrm{op}}$ sends a height-$2$ core with level vector $(\#\mathrm{mxl}(X),\#B_X,\#\mathrm{mnl}(X))=(p,m,q)$ to one with level vector $(q,m,p)$, and by Theorem~\ref{thm:mccord}(2) it preserves the weak homotopy type. A core that is homeomorphic to its opposite (a \emph{self-dual} core) is counted once. A core that is not homeomorphic to its opposite is counted twice: once for itself and once for its opposite. When $p\neq q$, a core and its opposite have different level vectors and so are never homeomorphic. Below we enumerate cores with $p\le q$ and then add the opposites, taking care not to count self-dual cores twice. In the figures, a panel labelled (x$^\ast$) shows the opposite of the poset in panel (x), relabelled so that the letters $a$, $b$, $c$ again denote maximal, intermediate and minimal points; a self-dual core is shown once and labelled (x) $=$ (x$^\ast$).
\end{remark}
\Needspace{12\baselineskip}
\begin{theorem}\label{thm:eight}
	The minimal finite models of the six spaces listed below have exactly eight points. Up to homeomorphism, their numbers are as follows, and they are the posets displayed in the indicated figures.
	\begin{center}
		\begin{tabular}{lcc}
			\toprule
			Space & Minimal finite models & Figure\\
			\midrule
			$S^{1}\vee S^{1}\vee S^{2}$ & $7$ & \ref{k}\\
			$S^{1}\vee S^{1}\vee S^{1}\vee S^{2}$ & $1$ & \ref{q}\\
			$S^{1}\vee S^{2}\vee S^{2}$ & $6$ & \ref{lprime}\\
			$S^{2}\vee S^{2}\vee S^{2}$ & $5$ & \ref{hprime}\\
			$S^{2}\vee S^{2}\vee S^{2}\vee S^{2}$ & $3$ & \ref{p}\\
				$S^{1}\vee S^{1}\vee S^{2}\vee S^{2}$ & $2$ & \ref{twotwo}\\
			\bottomrule
		\end{tabular}
	\end{center}
\end{theorem}

\begin{figure}[htbp]
	\centering
	\begin{adjustbox}{max width=\linewidth, max totalheight=0.60\textheight}\begin{tikzpicture}
		[acteur/.style={circle, fill=black,thick, inner sep=2pt, minimum size=0.2cm}]
		\begin{scope}[shift={(0,0)}]
			\node (1) at (0,0) [acteur,label=below:$c_1$]{};
			\node (2) at (1,0) [acteur,label=below:$c_2$]{};
			\node (3) at (2,0) [acteur,label=below:$c_3$]{};
			\node (8) at (3,0) [acteur,label=below:$c_4$]{};
			\node (4) at (0,1) [acteur,label=left:$b_1$]{};
			\node (5) at (1,1) [acteur,label=left:$b_2$]{};
			\node (6) at (1,2) [acteur,label=above:$a_1$]{};
			\node (7) at (2,2) [acteur,label=above:$a_2$]{};
			\draw[-, thick, red] (1) -- (4);
			\draw[-, thick, red] (1) -- (5);
			\draw[-, thick, red] (2) -- (4);
			\draw[-, thick, red] (2) -- (5);
			\draw[-, thick, red] (3) -- (6);
			\draw[-, thick, red] (3) -- (7);
			\draw[-, thick, red] (4) -- (6);
			\draw[-, thick, red] (4) -- (7);
			\draw[-, thick, red] (5) -- (6);
			\draw[-, thick, red] (5) -- (7);
			\draw[-, thick, red] (8) -- (6);
			\draw[-, thick, red] (8) -- (7);
			\node at (1.5,-0.8) {(a)};
		\end{scope}
		\begin{scope}[shift={(5,0)}]
			\node (1) at ( 1,0) [acteur,label=below:$c_1$]{};
			\node (2) at ( 2,0) [acteur,label=below:$c_2$]{};
			\node (3) at (0,1) [acteur][label=left:$b_1$]{};
			\node (4) at (1,1) [acteur][label=left:$b_2$]{};
			\node (5) at (2,2) [acteur,label=above:$a_3$]{};
			\node (6) at (3,2) [acteur,label=above:$a_4$]{};
			\node (7) at (0,2) [acteur][label=above:$a_1$]{};
			\node (8) at (1,2) [acteur][label=above:$a_2$]{};
			\draw [-, thick, red] (1) -- (3) (1) -- (4) (1) -- (5) (1) -- (6);
			\draw [-, thick, red] (2) -- (3) (2) -- (4) (2) -- (5) (2) -- (6);
			\draw [-, thick, red] (3) -- (7) (3) -- (8);
			\draw [-, thick, red] (4) -- (7) (4) -- (8);
			\node at (1.5,-0.8) {(a$^\ast$) Dual of (a)};
		\end{scope}
		\begin{scope}[shift={(10,0)}]
			\node (1) at (0,0) [acteur,label=below:$c_1$]{};
			\node (2) at (1,0) [acteur,label=below:$c_2$]{};
			\node (3) at (0,1) [acteur,label=left:$b_1$]{};
			\node (4) at (.8,1) [acteur,label=left:$b_2$]{};
			\node (8) at (2.2,0) [acteur,label=below:$c_3$]{};
			\node (5) at (0,2) [acteur,label=above:$a_1$]{};
			\node (6) at (1,2) [acteur,label=above:$a_2$]{};
			\node (7) at (2.2,2) [acteur,label=above:$a_3$]{};
			\draw[-, thick, red] (1)--(3) (1)--(4) (2)--(3) (2)--(4);
			\draw[-, thick, red] (3)--(5) (3)--(6) (4)--(5) (4)--(6);
			\draw[-, thick, red] (8)--(5) (8)--(6) (7)--(1) (7)--(2);
			\node at (1.1,-0.8) {(b) $=$ (b$^\ast$)};
		\end{scope}
		\begin{scope}[shift={(0,-5)}]
			\node (1) at (0,0) [acteur,label=below:$c_1$]{};
			\node (2) at (1,0) [acteur,label=below:$c_2$]{};
			\node (3) at (0,1) [acteur,label=left:$b_1$]{};
			\node (4) at (.8,1) [acteur,label=left:$b_2$]{};
			\node (8) at (2.2,0) [acteur,label=below:$c_3$]{};
			\node (5) at (0,2) [acteur,label=above:$a_1$]{};
			\node (6) at (1,2) [acteur,label=above:$a_2$]{};
			\node (7) at (2.2,2) [acteur,label=above:$a_3$]{};
			\draw[-, thick, red] (1)--(3) (1)--(4) (2)--(3) (2)--(4);
			\draw[-, thick, red] (3)--(5) (3)--(6) (4)--(5) (4)--(6);
			\draw[-, thick, red] (8)--(5) (8)--(6) (8)--(7) (7)--(2);
			\node at (1.2,-0.8) {(c)};
		\end{scope}
		\begin{scope}[shift={(5,-5)}]
			\node (1) at (0,2) [acteur,label=above:$a_1$]{};
			\node (2) at (1,2) [acteur,label=above:$a_2$]{};
			\node (3) at (0,1) [acteur,label=left:$b_1$]{};
			\node (4) at (.8,1) [acteur,label=left:$b_2$]{};
			\node (8) at (2.2,2) [acteur,label=above:$a_3$]{};
			\node (5) at (0,0) [acteur,label=below:$c_1$]{};
			\node (6) at (1,0) [acteur,label=below:$c_2$]{};
			\node (7) at (2.2,0) [acteur,label=below:$c_3$]{};
			\draw[-, thick, red] (1)--(3) (1)--(4) (2)--(3) (2)--(4);
			\draw[-, thick, red] (3)--(5) (3)--(6)
			(4)--(5) (4)--(6);
			\draw[-, thick, red] (8)--(5) (8)--(6)
			(8)--(7) (7)--(2);
			\node at (1.2,-0.8) {(c$^\ast$) Dual of (c)};
		\end{scope}
		\begin{scope}[shift={(10,-5)}]
			\node (1) at (0,0) [acteur,label=below:$c_1$]{};
			\node (2) at (1,0) [acteur,label=below:$c_2$]{};
			\node (3) at (0,1) [acteur,label=left:$b_1$]{};
			\node (4) at (.8,1) [acteur,label=left:$b_2$]{};
			\node (8) at (2.2,0) [acteur,label=below:$c_3$]{};
			\node (5) at (0,2) [acteur,label=above:$a_1$]{};
			\node (6) at (1,2) [acteur,label=above:$a_2$]{};
			\node (7) at (2.2,2) [acteur,label=above:$a_3$]{};
			\draw[-, thick, red] (1)--(3) (1)--(4) (2)--(3) (2)--(4);
			\draw[-, thick, red] (3)--(5) (3)--(6) (4)--(5) (4)--(6);
			\draw[-, thick, red] (8)--(5) (8)--(6) (8)--(7) (7)--(4);
			\node at (1.2,-0.8) {(d)};
		\end{scope}
		\begin{scope}[shift={(0,-10)}]
			\node (1) at (0,2) [acteur,label=above:$a_1$]{};
			\node (2) at (1,2) [acteur,label=above:$a_2$]{};
			\node (3) at (0,1) [acteur,label=left:$b_1$]{};
			\node (4) at (.8,1) [acteur,label=left:$b_2$]{};
			\node (8) at (2.2,2) [acteur,label=above:$a_3$]{};
			\node (5) at (0,0) [acteur,label=below:$c_1$]{};
			\node (6) at (1,0) [acteur,label=below:$c_2$]{};
			\node (7) at (2.2,0) [acteur,label=below:$c_3$]{};
			\draw[-, thick, red] (1)--(3) (1)--(4) (2)--(3) (2)--(4);
			\draw[-, thick, red] (3)--(5) (3)--(6)
			(4)--(5) (4)--(6);
			\draw[-, thick, red] (8)--(5) (8)--(6)
			(8)--(7) (7)--(4);
			\node at (1.2,-0.8) {(d$^\ast$) Dual of (d)};
		\end{scope}
	\end{tikzpicture}\end{adjustbox}
	\vspace{.2cm} %
	\caption{Minimal finite models of $S^1 \vee S^1\vee S^2$}
	\label{k}
\end{figure}

\begin{figure}[htbp]
	\centering
	\begin{adjustbox}{max width=\linewidth, max totalheight=0.60\textheight}\begin{tikzpicture}
		[acteur/.style={circle, fill=black,thick, inner sep=2pt, minimum size=0.2cm}]
		\node (1) at (0,2) [acteur,label=above:$a_1$]{};
		\node (2) at (1,2) [acteur,label=above:$a_2$]{};
		\node (3) at (0,1) [acteur,label=left:$b_1$]{};
		\node (4) at (1,1) [acteur,label=left:$b_2$]{};
		\node (8) at (2.5,2) [acteur,label=above:$a_3$]{};
		\node (5) at (0.5,0) [acteur,label=below:$c_1$]{};
		\node (6) at (1.5,0) [acteur,label=below:$c_2$]{};
		\node (7) at (2.5,0) [acteur,label=below:$c_3$]{};
		\draw[-, thick, red] (1)--(3) (1)--(4) (1)--(7)
		(2)--(3) (2)--(4);
		\draw[-, thick, red] (3)--(5) (3)--(6)
		(4)--(5) (4)--(6) (2)--(7);
		\draw[-, thick, red] (8)--(5) (8)--(6) (8)--(7);
	\end{tikzpicture}\end{adjustbox}
	\vspace{.2cm} %
	\caption{Minimal finite model of $S^1 \vee S^1\vee S^1\vee S^2$}
	\label{q}
\end{figure}

\begin{figure}[htbp]
	\centering
	\begin{adjustbox}{max width=\linewidth, max totalheight=0.60\textheight}\begin{tikzpicture}[acteur/.style={circle, fill=black, thick, inner sep=2pt, minimum size=0.2cm}]
		
		\begin{scope}[shift={(0,0)}]
			\node (1) at (0,0) [acteur,label=below:$c_1$]{};
			\node (2) at (1,0) [acteur,label=below:$c_2$]{};
			\node (8) at (2.5,0) [acteur,label=below:$c_3$]{};
			\node (3) at (0,1) [acteur,label=left:$b_1$]{};
			\node (4) at (1,1) [acteur,label=left:$b_2$]{};
			\node (5) at (2,1) [acteur,label=right:$b_3$]{};
			\node (6) at (0,2) [acteur,label=above:$a_1$]{};
			\node (7) at (1,2) [acteur,label=above:$a_2$]{};
			\draw[-, thick, red] (1)--(3) (1)--(4) (1)--(5) (2)--(3) (2)--(4) (2)--(5);
			\draw[-, thick, red] (3)--(6) (3)--(7) (4)--(6) (4)--(7) (5)--(6) (5)--(7);
			\draw[-, thick, red] (6)--(8) (7)--(8);
			\node at (1,-0.8) {(a)};
		\end{scope}
		
		\begin{scope}[shift={(5,0)}]
			\node (1) at (0,0) [acteur,label=below:$c_1$]{};
			\node (2) at (1,0) [acteur,label=below:$c_2$]{};
			\node (3) at (0,1) [acteur,label=left:$b_1$]{};
			\node (4) at (1,1) [acteur,label=left:$b_2$]{};
			\node (5) at (2,1) [acteur,label=left:$b_3$]{};
			\node (8) at (2.5,2) [acteur,label=above:$a_3$]{};
			\node (6) at (0,2) [acteur,label=above:$a_1$]{};
			\node (7) at (1,2) [acteur,label=above:$a_2$]{};
			\draw[-, thick, red] (1)--(3) (1)--(4) (1)--(5) (2)--(3) (2)--(4) (2)--(5);
			\draw[-, thick, red] (3)--(6) (3)--(7) (4)--(6) (4)--(7) (5)--(6) (5)--(7);
			\draw[-, thick, red] (1)--(8) (2)--(8);
			\node at (1,-0.8) {(a$^\ast$) Dual of (a)};
		\end{scope}
		
		\begin{scope}[shift={(10,0)}]
			\node (5) at (0,0) [acteur,label=below:$c_1$]{};
			\node (6) at (1,0) [acteur,label=below:$c_2$]{};
			\node (7) at (2,0) [acteur,label=below:$c_3$]{};
			\node (8) at (3,0) [acteur,label=below:$c_4$]{};
			\node (3) at (0,1) [acteur,label=left:$b_1$]{};
			\node (4) at (1,1) [acteur,label=left:$b_2$]{};
			\node (1) at (0,2) [acteur,label=above:$a_1$]{};
			\node (2) at (1,2) [acteur,label=above:$a_2$]{};
			
			\draw[-, thick, red] (3)--(1) (4)--(1) (3)--(2) (4)--(2);
			\draw[-, thick, red] (5)--(3) (6)--(3) (7)--(3) (5)--(4) (6)--(4) (7)--(4);
			\draw[-, thick, red] (1)--(8) (2)--(8);
			\node at (1,-0.8) {(b)};
		\end{scope}
		
		\begin{scope}[shift={(0,-5)}]
			\node (1) at (0,0) [acteur,label=below:$c_1$]{};
			\node (2) at (1,0) [acteur,label=below:$c_2$]{};
			\node (3) at (0,1) [acteur,label=left:$b_1$]{};
			\node (4) at (1,1) [acteur,label=left:$b_2$]{};
			\node (5) at (0,2) [acteur,label=above:$a_1$]{};
			\node (6) at (1,2) [acteur,label=above:$a_2$]{};
			\node (7) at (2,2) [acteur,label=above:$a_3$]{};
			\node (8) at (3,2) [acteur,label=above:$a_4$]{};
			\draw[-, thick, red] (1)--(3) (1)--(4) (2)--(3) (2)--(4);
			\draw[-, thick, red] (3)--(5) (3)--(6) (3)--(7) (4)--(5) (4)--(6) (4)--(7);
			\draw[-, thick, red] (8)--(1) (8)--(2);
			\node at (1,-0.8) {(b$^\ast$) Dual of (b)};
		\end{scope}
		
		\begin{scope}[shift={(5,-5)}]
			\node (1) at (0,0) [acteur,label=below:$c_1$]{};
			\node (2) at (1,0) [acteur,label=below:$c_2$]{};
			\node (3) at (0,1) [acteur,label=left:$b_1$]{};
			\node (4) at (1,1) [acteur,label=left:$b_2$]{};
			\node (8) at (2.5,0) [acteur,label=below:$c_3$]{};
			\node (5) at (0,2) [acteur,label=above:$a_1$]{};
			\node (6) at (1,2) [acteur,label=above:$a_2$]{};
			\node (7) at (2,2) [acteur,label=above:$a_3$]{};
			\draw[-, thick, red] (1)--(3) (1)--(4) (2)--(3) (2)--(4);
			\draw[-, thick, red] (3)--(5) (3)--(6) (3)--(7) (4)--(5) (4)--(6) (4)--(7);
			\draw[-, thick, red] (8)--(5) (8)--(6);
			\node at (1,-0.8) {(c)};
		\end{scope}
		
		\begin{scope}[shift={(10,-5)}]
			\node (5) at (0,0) [acteur,label=below:$c_1$]{};
			\node (6) at (1,0) [acteur,label=below:$c_2$]{};
			\node (7) at (2,0) [acteur,label=below:$c_3$]{};
			\node (3) at (0,1) [acteur,label=left:$b_1$]{};
			\node (4) at (1,1) [acteur,label=left:$b_2$]{};
			\node (8) at (2.5,2) [acteur,label=above:$a_3$]{};
			\node (1) at (0,2) [acteur,label=above:$a_1$]{};
			\node (2) at (1,2) [acteur,label=above:$a_2$]{};
			\draw[-, thick, red] (3)--(1) (4)--(1) (3)--(2) (4)--(2);
			\draw[-, thick, red] (5)--(3) (6)--(3) (7)--(3) (5)--(4) (6)--(4) (7)--(4);
			\draw[-, thick, red] (5)--(8) (6)--(8);
			\node at (1,-0.8) {(c$^\ast$) Dual of (c)};
		\end{scope}
		
	\end{tikzpicture}\end{adjustbox}
	\caption{Minimal finite models of $S^1 \vee S^2\vee S^2$}
	\label{lprime}
\end{figure}

\begin{figure}[htbp]
	\centering
	\begin{adjustbox}{max width=\linewidth, max totalheight=0.60\textheight}\begin{tikzpicture}
		[acteur/.style={circle, fill=black,thick, inner sep=2pt, minimum size=0.2cm}]
		\begin{scope}[shift={(0,0)}]
			\node (5) at (0,0) [acteur,label=below:$c_1$]{};
			\node (6) at (1,0) [acteur,label=below:$c_2$]{};
			\node (7) at (2,0) [acteur,label=below:$c_3$]{};
			\node (8) at (3,0) [acteur,label=below:$c_4$]{};
			\node (3) at (0,1) [acteur,label=left:$b_1$]{};
			\node (4) at (1,1) [acteur,label=right:$b_2$]{};
			\node (1) at (0,2) [acteur,label=above:$a_1$]{};
			\node (2) at (1,2) [acteur,label=above:$a_2$]{};
			\draw[-, thick, red]
			(5)--(3) (5)--(4) (6)--(3) (6)--(4)
			(7)--(3) (7)--(4) (8)--(3) (8)--(4)
			(3)--(1) (3)--(2) (4)--(1) (4)--(2);
			\node at (1.5,-0.8) {(a)};
		\end{scope}
		
		\begin{scope}[shift={(5,0)}]
			\node (1) at (0,0) [acteur,label=below:$c_1$]{};
			\node (2) at (1,0) [acteur,label=below:$c_2$]{};
			\node (3) at (0,1) [acteur,label=left:$b_1$]{};
			\node (4) at (1,1) [acteur,label=right:$b_2$]{};
			\node (5) at (0,2) [acteur,label=above:$a_1$]{};
			\node (6) at (1,2) [acteur,label=above:$a_2$]{};
			\node (7) at (2,2) [acteur,label=above:$a_3$]{};
			\node (8) at (3,2) [acteur,label=above:$a_4$]{};
			\draw[-, thick, red] (1)--(3) (1)--(4) (2)--(3) (2)--(4);
			\draw[-, thick, red] (3)--(5) (3)--(6) (3)--(7) (3)--(8);
			\draw[-, thick, red] (4)--(5) (4)--(6) (4)--(7) (4)--(8);
			\node at (1.5,-0.8) {(a$^\ast$) Dual of (a)};
		\end{scope}
		
		\begin{scope}[shift={(10,0)}]
			\node (1) at (0,0) [acteur,label=below:$c_1$]{};
			\node (2) at (1,0) [acteur,label=below:$c_2$]{};
			\node (3) at (0,1) [acteur,label=left:$b_1$]{};
			\node (4) at (1,1) [acteur,label=left:$b_2$]{};
			\node (5) at (2,1) [acteur,label=left:$b_3$]{};
			\node (8) at (3,1) [acteur,label=right:$b_4$]{};
			\node (6) at (0,2) [acteur,label=above:$a_1$]{};
			\node (7) at (1,2) [acteur,label=above:$a_2$]{};
			\draw[-, thick, red] (1)--(3) (1)--(4) (1)--(5);
			\draw[-, thick, red] (2)--(3) (2)--(4) (2)--(5);
			\draw[-, thick, red] (3)--(6) (3)--(7);
			\draw[-, thick, red] (4)--(6) (4)--(7);
			\draw[-, thick, red] (5)--(6) (5)--(7);
			\draw[-, thick, red] (8)--(6) (8)--(7) (8)--(1) (8)--(2);
			\node at (1.5,-0.8) {(b) $=$ (b$^\ast$)};
		\end{scope}
		
		\begin{scope}[shift={(0,-5)}]
			\node (1) at (0,0) [acteur,label=below:$c_1$]{};
			\node (2) at (1,0) [acteur,label=below:$c_2$]{};
			\node (3) at (0,1) [acteur,label=left:$b_1$]{};
			\node (4) at (1,1) [acteur,label=left:$b_2$]{};
			\node (8) at (2,0) [acteur,label=below:$c_3$]{};
			\node (5) at (0,2) [acteur,label=above:$a_1$]{};
			\node (6) at (1,2) [acteur,label=above:$a_2$]{};
			\node (7) at (2,1) [acteur,label=right:$b_3$]{};
			\draw[-, thick, red] (1)--(3) (1)--(4) (1)--(7) (2)--(3) (2)--(4);
			\draw[-, thick, red] (3)--(5) (3)--(6) (2)--(7) (8)--(7) (4)--(5)  (4)--(6);
			\draw[-, thick, red] (5)--(7) (6)--(7) (8)--(4);
			\node at (1,-0.8) {(c)};
		\end{scope}
		
		\begin{scope}[shift={(5,-5)}] %
			\node (1) at (0,2) [acteur,label=above:$a_1$]{};
			\node (2) at (1,2) [acteur,label=above:$a_2$]{};
			\node (8) at (2,2) [acteur,label=above:$a_3$]{};
			\node (3) at (0,1) [acteur,label=left:$b_1$]{};
			\node (4) at (1,1) [acteur,label=left:$b_2$]{};
			\node (7) at (2,1) [acteur,label=right:$b_3$]{};
			\node (5) at (0,0) [acteur,label=below:$c_1$]{};
			\node (6) at (1,0) [acteur,label=below:$c_2$]{};
			\draw[-, thick, red] (3)--(1) (4)--(1) (7)--(1) (3)--(2) (4)--(2);
			\draw[-, thick, red] (5)--(3) (6)--(3) (7)--(2) (7)--(8) (5)--(4) (6)--(4);
			\draw[-, thick, red] (7)--(5) (7)--(6) (4)--(8);
			\node at (1,-0.8) {(c$^\ast$) Dual of (c)};
		\end{scope}
	\end{tikzpicture}\end{adjustbox}
	\caption{Minimal finite models of $S^2\vee S^2\vee S^2$}
	\label{hprime}
\end{figure}

\begin{figure}[htbp]
	\centering
	\begin{adjustbox}{max width=\linewidth, max totalheight=0.60\textheight}\begin{tikzpicture}
		[acteur/.style={circle, fill=black,thick, inner sep=2pt, minimum size=0.2cm}]
		\begin{scope}[shift={(0,0)}]
			\node (1) at (0,0) [acteur,label=below:$c_1$]{};
			\node (2) at (1,0) [acteur,label=below:$c_2$]{};
			\node (3) at (0,1) [acteur,label=left:$b_1$]{};
			\node (4) at (1,1) [acteur,label=left:$b_2$]{};
			\node (8) at (2,0) [acteur,label=below:$c_3$]{};
			\node (5) at (0,2) [acteur,label=above:$a_1$]{};
			\node (6) at (1,2) [acteur,label=above:$a_2$]{};
			\node (7) at (2,1) [acteur,label=right:$b_3$]{};
			\draw[-, thick, red] (1)--(3) (1)--(4) (1)--(7) (2)--(3) (2)--(4);
			\draw[-, thick, red] (3)--(5) (3)--(6) (3)--(8) (2)--(7) (8)--(7) (4)--(5)  (4)--(6);
			\draw[-, thick, red] (5)--(7) (6)--(7) (8)--(4);
			\node at (1,-0.8) {(a)};
		\end{scope}
		
		\begin{scope}[shift={(5,0)}] %
			\node (1) at (0,2) [acteur,label=above:$a_1$]{};
			\node (2) at (1,2) [acteur,label=above:$a_2$]{};
			\node (8) at (2,2) [acteur,label=above:$a_3$]{};
			\node (3) at (0,1) [acteur,label=left:$b_1$]{};
			\node (4) at (1,1) [acteur,label=left:$b_2$]{};
			\node (7) at (2,1) [acteur,label=right:$b_3$]{};
			\node (5) at (0,0) [acteur,label=below:$c_1$]{};
			\node (6) at (1,0) [acteur,label=below:$c_2$]{};
			\draw[-, thick, red] (3)--(1) (4)--(1) (7)--(1) (3)--(2) (4)--(2);
			\draw[-, thick, red] (5)--(3) (6)--(3) (8)--(3) (7)--(2) (7)--(8) (5)--(4) (6)--(4);
			\draw[-, thick, red] (7)--(5) (7)--(6) (4)--(8);
			\node at (1,-0.8) {(a$^\ast$) Dual of (a)};
		\end{scope}
		
		\begin{scope}[shift={(10,0)}]
			\node (1) at (0,0) [acteur,label=below:$c_1$]{};
			\node (2) at (1,0) [acteur,label=below:$c_2$]{};
			\node (3) at (0,1) [acteur,label=left:$b_1$]{};
			\node (4) at (1,1) [acteur,label=left:$b_2$]{};
			\node (8) at (2,0) [acteur,label=below:$c_3$]{};
			\node (5) at (0,2) [acteur,label=above:$a_1$]{};
			\node (6) at (1,2) [acteur,label=above:$a_2$]{};
			\node (7) at (2,2) [acteur,label=above:$a_3$]{};
			\draw[-, thick, red] (1)--(3) (1)--(4) (2)--(3) (2)--(4);
			\draw[-, thick, red] (3)--(5) (3)--(6) (3)--(7) (4)--(5)  (4)--(6) (4)--(7);
			\draw[-, thick, red] (8)--(3) (8)--(4);
			\node at (1.5,-0.8) {(b) $=$ (b$^\ast$)};
		\end{scope}
	\end{tikzpicture}\end{adjustbox}
	\caption{Minimal finite models of $S^2\vee S^2\vee S^2\vee S^2$}
	\label{p}
\end{figure}

\begin{figure}[htbp]
	\centering
	\begin{adjustbox}{max width=\linewidth, max totalheight=0.60\textheight}\begin{tikzpicture}
		[acteur/.style={circle, fill=black,thick, inner sep=2pt, minimum size=0.2cm}]
		\begin{scope}[shift={(0,0)}]
			\node (1) at (0,0) [acteur,label=below:$c_1$]{};
			\node (2) at (1,0) [acteur,label=below:$c_2$]{};
			\node (3) at (0,1) [acteur,label=left:$b_1$]{};
			\node (4) at (1,1) [acteur,label=left:$b_2$]{};
			\node (8) at (2.5,0) [acteur,label=below:$c_3$]{};
			\node (5) at (0.5,2) [acteur,label=above:$a_1$]{};
			\node (6) at (1.5,2) [acteur,label=above:$a_2$]{};
			\node (7) at (2.5,2) [acteur,label=above:$a_3$]{};
			\draw[-, thick, red] (1)--(3) (1)--(4) (4)--(7) (2)--(3) (2)--(4);
			\draw[-, thick, red] (3)--(5) (3)--(6) (4)--(5)  (4)--(6) (3)--(7);
			\draw[-, thick, red] (8)--(5) (8)--(6) (8)--(7);
			\node at (1.5,-0.8) {(a)};
		\end{scope}
		\begin{scope}[shift={(5,0)}]
			\node (1) at (0,2) [acteur,label=above:$a_1$]{};
			\node (2) at (1,2) [acteur,label=above:$a_2$]{};
			\node (8) at (2.5,2) [acteur,label=above:$a_3$]{};
			\node (3) at (0,1) [acteur,label=left:$b_1$]{};
			\node (4) at (1,1) [acteur,label=left:$b_2$]{};
			\node (5) at (0.5,0) [acteur,label=below:$c_1$]{};
			\node (6) at (1.5,0) [acteur,label=below:$c_2$]{};
			\node (7) at (2.5,0) [acteur,label=below:$c_3$]{};
			\draw[-, thick, red] (3)--(1) (4)--(1) (7)--(4) (3)--(2) (4)--(2);
			\draw[-, thick, red] (5)--(3) (6)--(3) (5)--(4) (6)--(4) (7)--(3);
			\draw[-, thick, red] (5)--(8) (6)--(8) (7)--(8);
			\node at (1.5,-0.8) {(a$^\ast$) Dual of (a)};
		\end{scope}
	\end{tikzpicture}\end{adjustbox}
	\caption{Minimal finite models of $S^1\vee S^1\vee S^2\vee S^2$}
	\label{twotwo}
\end{figure}

\begin{proof}
	By Lemma~\ref{lem:lb8}, every finite model of the six spaces has at least eight points, and every eight-point minimal finite model of them is a connected eight-point core of height $2$. We classify these cores and identify the minimal finite models among them; since each of the six spaces turns out to have eight-point models, its minimal finite models have exactly eight points.
	
	Let $X$ be a connected eight-point core of height $2$. Replacing $X$ by its opposite $X^{\mathrm{op}}$ if necessary (same weak homotopy type, Theorem~\ref{thm:mccord}(2)), we may assume
	\[
	\#\mathrm{mxl}(X)\leq \#\mathrm{mnl}(X).
	\]
	An eight-point core is not contractible, since the only contractible core is a point; so $\#\mathrm{mxl}(X)\geq 2$ and $\#\mathrm{mnl}(X)\geq 2$ by Remark~\ref{rem:mxlmnl}. Since $X$ is connected with more than one point, $\mathrm{mxl}(X)\cap \mathrm{mnl}(X)=\varnothing$, and as $\mathrm{h}(X)=2$, $B_X\neq \varnothing$. The combinatorial core criterion (Lemma~\ref{lem:crit}) and the homotopy-type tools (Lemmas~\ref{lem:tripartite},~\ref{lem:pifree},~\ref{lem:twomax} and~\ref{lem:twomid}) are used freely below. In every case below, $X$ has one or two intermediate points or two maximal points, so Lemma~\ref{lem:onemid} or Lemma~\ref{lem:pifree} applies: $\mathcal{K}(X)$ is homotopy equivalent to a wedge of $\beta_1$ circles and $\beta_2$ two-spheres, so $(\beta_1,\beta_2)$ determines the weak homotopy type. We shall not repeat this.
	
	Since $\#X=8$, the possible distributions of points among the three levels are determined by
	\[
	\#X = \#\mathrm{mxl}(X) + \#B_X + \#\mathrm{mnl}(X) = 8,
	\]
	so that $\#B_X \in \{1,2,3,4\}$.
	We analyse these cases separately.
	
	\medskip
	\noindent\textbf{Case 1: $\#B_X=1$.}
	
	By Lemma~\ref{lem:onemid}, every core with a single intermediate point is weak homotopy equivalent to a wedge of circles, so $H_2(\mathcal{K}(X))=0$ and $X$ models none of the six spaces of Lemma~\ref{lem:lb8} (all of which have $H_2\neq 0$). For completeness we record the cores that arise. There are two distributions for the remaining points:
	\[
	(\#\mathrm{mxl}(X), \#\mathrm{mnl}(X)) = (2,5) \quad \text{or} \quad (3,4).
	\]
	
	\smallskip
	\noindent\emph{\textbf{Subcase 1.1:}} $\#\mathrm{mxl}(X)=2$, $\#\mathrm{mnl}(X)=5$.
	
	Let
	\[
	\mathrm{mxl}(X)=\{a_1,a_2\}, \quad B_X=\{b_1\}, \quad \mathrm{mnl}(X)=\{c_1,c_2,c_3,c_4,c_5\}.
	\]
	By Lemma~\ref{lem:co}, $b_1$ lies below both maximal points, and at least two minimal points lie below $b_1$. For a minimal point $c<b_1$, the set $\hat F_c$ consists of $b_1$ together with maximal points, all of which lie above $b_1$. Hence $\hat F_c$ has minimum $b_1$, and $c$ is an up beat point. So this subcase yields no core at all.
	
	\smallskip
	\noindent\emph{\textbf{Subcase 1.2:}} $\#\mathrm{mxl}(X)=3$, $\#\mathrm{mnl}(X)=4$.
	
	Let
	\[
	\mathrm{mxl}(X)=\{a_1,a_2, a_3\}, \quad B_X=\{b_1\}, \quad \mathrm{mnl}(X)=\{c_1,c_2,c_3,c_4\},
	\]
	and let $u=\#(\hat{F}_{b_1}\cap \mathrm{mxl}(X))$ and $d=\#(\hat{U}_{b_1}\cap \mathrm{mnl}(X))$. By Lemma~\ref{lem:co}, $u\ge2$ and $d\ge2$. By Lemma~\ref{lem:onemid}, $X\whe\bigvee_k S^1$ with $k=\operatorname{rank}H_1(\mathcal{K}(X))$. Let $P$ be the number of comparable pairs $(c,a)$ with $c$ minimal and $a$ maximal, so that $P\le12$. The order complex has $8$ vertices, $d+u+P$ edges and $du$ triangles, so
	\[
	k=1-\chi(\mathcal{K}(X))=P-6-(d-1)(u-1)\le 5 .
	\]
	By Corollary~\ref{cor:barmak82}, $\bigvee_k S^1$ with $k\le5$ has a finite model with at most seven points, so none of these cores is a minimal finite model.
	
	For completeness we list them. As in Case~1 of the proof of Theorem~\ref{thm:seven}, $b_1$ cannot lie below all three maximal points, since otherwise the minimal points below $b_1$ would be beat points. So $u=2$, say $b_1<a_1,a_2$. Then every minimal point below $b_1$ must be covered by $a_3$, and each of $a_1$ and $a_2$ must cover a minimal point not below $b_1$; hence $d\in\{2,3\}$. Let $T$ be the set of the $4-d$ minimal points not below $b_1$. Each $t\in T$ is covered directly by a set $N(t)\subseteq\{a_1,a_2,a_3\}$ of at least two maximal points, and $a_1$ and $a_2$ must each belong to some $N(t)$. Conversely, every such choice gives a core by Lemma~\ref{lem:crit}. Up to the symmetries $a_1\leftrightarrow a_2$ and permutations of $T$, there are two cores with $d=3$ (with $N(t)=\{a_1,a_2\}$ or $N(t)=\{a_1,a_2,a_3\}$) and six cores with $d=2$. Since $P=3d+\sum_{t\in T}\#N(t)$, here $k=\sum_{t\in T}\#N(t)+2d-5$. Altogether there are eight such cores: four with $k=3$, three with $k=4$ and one with $k=5$. Together with their duals they form Block~I of Table~\ref{tab:nonmin}.
	
	\medskip
	\noindent\textbf{Case 2: $\#B_X=2$.}
	
	Here, there are two possible distributions for the remaining points:
	\[
	(\#\mathrm{mxl}(X), \#\mathrm{mnl}(X)) = (2,4) \quad \text{or} \quad (3,3).
	\]
	
	\smallskip
	\noindent\emph{\textbf{Subcase 2.1:}} $\#\mathrm{mxl}(X)=2$, $\#\mathrm{mnl}(X)=4$.
	
	Let
	\[
	\mathrm{mxl}(X)=\{a_1,a_2\}, \quad B_X=\{b_1,b_2\}, \quad \mathrm{mnl}(X)=\{c_1,c_2,c_3,c_4\}.
	\]
	By Lemma~\ref{lem:co}, each $b_i$ lies below both maximal points, so the hypotheses of Lemma~\ref{lem:twomax} hold with $m=2$ and $q=4$. An inner minimal point is below both maximal points and so has no direct covers; by Lemma~\ref{lem:crit}(2) it lies below at least two intermediate points, hence below both. So the graph $G$ is the complete bipartite graph $K_{s,2}$, where $s$ is the number of inner minimal points. Each $b_i$ lies above at least two minimal points (Lemma~\ref{lem:crit}(1)), all of them inner, so $2\le s\le 4$. Thus $E=2s$ and $\kappa=1$, and by Lemma~\ref{lem:twomax}
	\[
	\beta_1=4-s,\qquad \beta_2=s-1 .
	\]
	By the determination principle (Lemma~\ref{lem:determination}), the poset is determined up to homeomorphism by $s\in\{2,3,4\}$: the inner minimal points are interchangeable, as are the outer ones, the two intermediate points and the two maximal points. Hence Subcase~2.1 yields exactly three cores.
	For $s=2$ we get $(\beta_1,\beta_2)=(2,1)$, so $X\whe S^{1}\vee S^{1}\vee S^{2}$ (Fig.~\ref{k}(a); its dual, with level vector $(4,2,2)$, is $(a^{\ast})$). For $s=3$ we get $(\beta_1,\beta_2)=(1,2)$, so $X\whe S^{1}\vee S^{2}\vee S^{2}$ (Fig.~\ref{lprime}(b); its dual is $(b^{\ast})$). For $s=4$ we get $(\beta_1,\beta_2)=(0,3)$, so $X\whe S^{2}\vee S^{2}\vee S^{2}$; here $X$ is the complete tripartite poset $T(2,2,4)$ (Fig.~\ref{hprime}(a); its dual is $(a^{\ast})$).
	
	\smallskip
	\noindent\emph{\textbf{Subcase 2.2:}} $\#\mathrm{mxl}(X)=3$, $\#\mathrm{mnl}(X)=3$.
	
	Let
	\[
	\mathrm{mxl}(X)=\{a_1,a_2, a_3\}, \quad B_X=\{b_1,b_2\}, \quad \mathrm{mnl}(X)=\{c_1,c_2,c_3\},
	\]
	and let $C_i=\hat U_{b_i}\cap\mathrm{mnl}(X)$ and $A_i=\hat F_{b_i}\cap\mathrm{mxl}(X)$. By Lemma~\ref{lem:co}, $|C_i|,|A_i|\in\{2,3\}$ for $i=1,2$. Since $X$ has exactly two intermediate points, Lemma~\ref{lem:twomid} applies with $p=q=3$:
	\[
	\beta_2=(|C_1\cap C_2|-1)(|A_1\cap A_2|-1),\qquad
	\beta_1=|C_1\cup C_2|+|A_1\cup A_2|+\delta-6,
	\]
	where $\delta=|\Delta|$ is the number of direct covers.
	
	\smallskip
	\noindent\emph{Reduction to symmetry orbits.} By the determination principle (Lemma~\ref{lem:determination}), a connected $(3,2,3)$ core is recovered from its datum $(C_1,C_2,A_1,A_2,\Delta)$. Two such cores are homeomorphic exactly when their data are related by the level-symmetry group $\mathrm{Sym}\{a_1,a_2,a_3\}\times\mathrm{Sym}\{b_1,b_2\}\times\mathrm{Sym}\{c_1,c_2,c_3\}$. Classifying these cores up to homeomorphism is therefore the same as listing the orbits, under this group, of the \emph{admissible} data, namely those for which $X$ is a core (Lemma~\ref{lem:crit}). We organise the enumeration by the isomorphism types of $(C_1,C_2)$ and of $(A_1,A_2)$. Each pair of subsets of size at least $2$ of a $3$-element set has one of four patterns (below). For each resulting pattern-pair we fix representatives $(C_1,C_2,A_1,A_2)$ and list all admissible $\Delta$ up to the residual symmetry, that is, up to the stabiliser of these representatives in the level-symmetry group. This is legitimate because, under the level-symmetry group above, each pattern-pair determines a single orbit of quadruples $(C_1,C_2,A_1,A_2)$, with one exception: the transposition $b_1\leftrightarrow b_2$ interchanges $(C_1,A_1)$ with $(C_2,A_2)$, and its effect on the $C_i$ can be undone by a permutation of $\mathrm{mnl}(X)$ in each of the patterns $\mathsf{F}$, $\mathsf{E}$ and $\mathsf{Z}$, and dually for the $A_i$; only $\mathsf{M}$, where $|C_1|\neq|C_2|$, is not symmetric in this way. Hence $(\mathsf{M},\mathsf{M})$ is the only pattern-pair whose quadruples form two orbits rather than one, according to whether the intermediate point with $C_i=\mathrm{mnl}(X)$ is the one with $A_i=\mathrm{mxl}(X)$; this is the case distinction made in item~5 below. The residual symmetries used are described in each item below, and the resulting orbit representatives are collected in Table~\ref{tab:323}; their weak homotopy types are given by Lemma~\ref{lem:twomid}. By the determination principle the $20$ representatives listed there are pairwise non-homeomorphic, and by the case analysis below they exhaust the admissible data, so Table~\ref{tab:323} is a complete and irredundant list.
	
	\smallskip
	\noindent\emph{Down- and up-patterns.} Two subsets $C_1,C_2$ of the three minimal points, each of size $2$ or $3$, have, up to relabelling, exactly one of the following four \emph{patterns}, determined by $(|C_1\cap C_2|,\,|C_1\cup C_2|)$:
	\[
	\begin{array}{ll}
	\mathsf{F}=(3,3): & C_1=C_2=\mathrm{mnl}(X);\\
	\mathsf{M}=(2,3): & \{C_1,C_2\}=\{\mathrm{mnl}(X),\{c_1,c_2\}\};\\
	\mathsf{E}=(2,2): & C_1=C_2=\{c_1,c_2\};\\
	\mathsf{Z}=(1,3): & \{C_1,C_2\}=\{\{c_1,c_2\},\{c_1,c_3\}\}.
	\end{array}
	\]
	Thus in pattern $\mathsf{F}$ both intermediate points lie above all minimal points, in pattern $\mathsf{M}$ exactly one of them does, in pattern $\mathsf{E}$ they lie above the same two minimal points, and in pattern $\mathsf{Z}$ they have exactly one minimal point in common below them. The same four \emph{up-patterns} arise for $(A_1,A_2)$. We write a pattern-pair as $(P,Q)$, where $P$ is the down-pattern and $Q$ the up-pattern; the opposite of a core with pattern-pair $(P,Q)$ has pattern-pair $(Q,P)$. A minimal point lying below $k$ intermediate points has $k$ intermediate covers, so by Lemma~\ref{lem:crit}(2) it needs at least $2-k$ direct covers, by maximal points that lie above none of those intermediate points; dually for maximal points. We use this to determine, for each pattern-pair, all cores and their types.
	
	\smallskip
	\noindent\emph{Vanishing of $\beta_2$.} If the down-pattern or the up-pattern is $\mathsf{Z}$, then $|C_1\cap C_2|=1$ or $|A_1\cap A_2|=1$, so $\beta_2=0$ and $X$ is weak homotopy equivalent to a wedge of circles. Such an $X$ models none of the six spaces of Lemma~\ref{lem:lb8}, all of which have $H_2\neq0$. The pattern-pairs $(\mathsf{Z},\mathsf{M})$ and $(\mathsf{Z},\mathsf{F})$, and their duals, admit no core. Indeed, in pattern $\mathsf{Z}$ each intermediate point lies above a minimal point that lies below no other intermediate point. If $b_i$ satisfies $A_i=\mathrm{mxl}(X)$, such a minimal point below $b_i$ lies below every maximal point through $b_i$, so $b_i$ is its only immediate successor and that minimal point is a beat point. The remaining pattern-pairs of this kind are $(\mathsf{Z},\mathsf{E})$, its dual $(\mathsf{E},\mathsf{Z})$, and $(\mathsf{Z},\mathsf{Z})$.
	For $(\mathsf{Z},\mathsf{E})$, take $C_1=\{c_1,c_2\}$, $C_2=\{c_1,c_3\}$ and $A_1=A_2=\{a_1,a_2\}$. The point $c_2$ lies below $b_1$ only, and its only possible direct cover is by $a_3$; so $c_2\prec a_3$ is forced, and likewise $c_3\prec a_3$. The direct cover $c_1\prec a_3$ is optional, and no other direct cover is possible. Hence $\delta\in\{2,3\}$ and $(\beta_1,\beta_2)=(\delta-1,0)$, giving two cores, of types $(1,0)$ and $(2,0)$. For $(\mathsf{Z},\mathsf{Z})$, take $C_1=\{c_1,c_2\}$, $C_2=\{c_1,c_3\}$, $A_1=\{a_1,a_2\}$ and $A_2=\{a_1,a_3\}$. The direct covers $c_2\prec a_3$ and $c_3\prec a_2$ are forced, and no others are possible, since every other pair $(c,a)$ is already comparable through $b_1$ or $b_2$. Hence $\delta=2$, giving a single core, of type $(2,0)$; it is self-dual.
	Together with their duals, these are five of the nine cores of Block~II of Table~\ref{tab:nonmin}; they are not minimal finite models.
	
	\smallskip
	\noindent\emph{Pattern-pairs with $\beta_2\ge1$.} It remains to treat the pattern-pairs in which both patterns lie in $\{\mathsf{F},\mathsf{M},\mathsf{E}\}$. We go through them up to duality. Throughout, ``a unique core'' (resp.\ ``$k$ cores'') means a single orbit (resp.\ $k$ orbits) of admissible $\Delta$ under the stabiliser of the chosen representatives $(C_1,C_2,A_1,A_2)$; each item below exhibits the orbit representatives and verifies the core condition by Lemma~\ref{lem:crit}.
	
	\begin{enumerate}
		\item \textbf{$(\mathsf{F},\mathsf{F})$.} Both intermediate points lie above all minimal points and below all maximal points, so $X=T(3,2,3)$, the complete tripartite poset, with $\delta=0$. By Lemma~\ref{lem:tripartite} (or by Lemma~\ref{lem:twomid}, which gives $\beta_2=4$ and $\beta_1=0$), $X\whe S^2\vee S^2\vee S^2\vee S^2$. This is the unique core here (Fig.~\ref{p}(b)).
		
		\item \textbf{$(\mathsf{F},\mathsf{M})$ and its dual $(\mathsf{M},\mathsf{F})$: no core.} Suppose the down-pattern is $\mathsf{F}$ and the up-pattern is $\mathsf{M}$, say $A_1=\mathrm{mxl}(X)$ and $A_2=\{a_1,a_2\}$. Then $a_3$ lies above $b_1$ only, and its only possible further immediate predecessor would be a direct cover $c\prec a_3$. But every minimal point lies below $b_1$ (down-pattern $\mathsf{F}$), and $b_1<a_3$, so $c<a_3$ holds through $b_1$ for every minimal point $c$, and no direct cover of $a_3$ is possible. Hence $a_3$ has the single immediate predecessor $b_1$ and is a beat point, so no core arises. The dual pattern-pair is treated in the same way.
		
		\item \textbf{$(\mathsf{E},\mathsf{F})$ and its dual $(\mathsf{F},\mathsf{E})$.} Take $C_1=C_2=\{c_1,c_2\}$ and $A_1=A_2=\mathrm{mxl}(X)$. Then $c_1,c_2$ lie below both intermediate points, hence below every maximal point, and so have no direct covers. The point $c_3$ lies below neither intermediate point, so by Lemma~\ref{lem:crit}(2) it has at least two direct covers, among $a_1,a_2,a_3$. Every maximal point already lies above both intermediate points and needs no direct cover. So $\Delta$ consists of the direct covers of $c_3$, and $\delta\in\{2,3\}$. Lemma~\ref{lem:twomid} gives $\beta_2=(2-1)(3-1)=2$ and $\beta_1=2+3+\delta-6=\delta-1$. Up to relabelling the maximal points, $\delta=2$ gives one core $\whe S^1\vee S^2\vee S^2$ (Fig.~\ref{lprime}(c)) and $\delta=3$ gives one core $\whe S^1\vee S^1\vee S^2\vee S^2$ (Fig.~\ref{twotwo}(a)); by Lemma~\ref{lem:lb8}, both are minimal finite models. With the dual pattern-pair this yields, in total, two cores $\whe S^1\vee S^2\vee S^2$ (Fig.~\ref{lprime}(c), (c$^\ast$)) and two cores $\whe S^1\vee S^1\vee S^2\vee S^2$ (Fig.~\ref{twotwo}(a), (a$^\ast$)); in each case the two cores are not homeomorphic, since their pattern-pairs $(\mathsf{E},\mathsf{F})$ and $(\mathsf{F},\mathsf{E})$ differ.
		
		\item \textbf{$(\mathsf{E},\mathsf{M})$ and its dual $(\mathsf{M},\mathsf{E})$.} Take $C_1=C_2=\{c_1,c_2\}$, and let one of $A_1,A_2$ be $\mathrm{mxl}(X)$ and the other $\{a_1,a_2\}$. Here $c_3$ lies below neither intermediate point, and $a_3$ lies above exactly one of them. A direct cover of $a_3$ can only come from $c_3$, since $c_1$ and $c_2$ lie below $a_3$ through the intermediate point $b_i$ with $A_i=\mathrm{mxl}(X)$. So the second immediate predecessor of $a_3$ forces $c_3\prec a_3$. The point $c_3$ needs a further direct cover, by $a_1$ or $a_2$; these two choices are equivalent under the stabiliser. Hence $\delta\in\{2,3\}$, $\beta_2=(1)(1)=1$ and $\beta_1=2+3+\delta-6=\delta-1$. The value $\delta=2$ gives one core $\whe S^2\vee S^1$, which is not a minimal finite model, and $\delta=3$ gives one core $\whe S^1\vee S^1\vee S^2$ (Fig.~\ref{k}(d)). With the dual pattern-pair, this yields two cores $\whe S^1\vee S^1\vee S^2$ (Fig.~\ref{k}(d), (d$^\ast$)) and two cores $\whe S^2\vee S^1$.
		
		\item \textbf{$(\mathsf{M},\mathsf{M})$.} Here one intermediate point $b_i$ lies above all minimal points ($C_i=\mathrm{mnl}(X)$) and one intermediate point $b_j$ lies below all maximal points ($A_j=\mathrm{mxl}(X)$). If $i=j$, the minimal point lying below $b_i$ only lies below every maximal point through $b_i$; so $b_i$ is its only immediate successor, that minimal point is a beat point, and there is no core. Otherwise we may take $C_1=\mathrm{mnl}(X)$, $C_2=\{c_1,c_2\}$, $A_1=\{a_1,a_2\}$ and $A_2=\mathrm{mxl}(X)$. Then $c_1,c_2$ lie below $b_2$ and hence below every maximal point, and $a_1,a_2$ lie above $b_1$ and hence above every minimal point. So the only possible direct cover is $c_3\prec a_3$, and it is forced, since $c_3$ lies below $b_1$ only and $a_3$ lies above $b_2$ only. Hence $\delta=1$, $\beta_2=(1)(1)=1$ and $\beta_1=3+3+1-6=1$. This is one self-dual core $\whe S^2\vee S^1$, which is not a minimal finite model.
		
		\item \textbf{$(\mathsf{E},\mathsf{E})$.} Take $C_1$ and $C_2$ both equal to $\{c_1,c_2\}$, and $A_1$ and $A_2$ both equal to $\{a_1,a_2\}$. The subposet on $\{c_1,c_2,b_1,b_2,a_1,a_2\}$ is a copy of $T(2,2,2)$, while $c_3$ lies below no intermediate point and $a_3$ above none. The only candidate direct covers are the five pairs involving $c_3$ or $a_3$, namely $c_3a_1$, $c_3a_2$, $c_3a_3$, $c_1a_3$ and $c_2a_3$. By Lemma~\ref{lem:crit}, a set $\Delta$ of such pairs gives a core if and only if $\Delta$ contains at least two of $c_3a_1,c_3a_2,c_3a_3$ and at least two of $c_1a_3,c_2a_3,c_3a_3$. The stabiliser of the chosen representatives is generated by $a_1\!\leftrightarrow\!a_2$, $b_1\!\leftrightarrow\!b_2$ and $c_1\!\leftrightarrow\!c_2$; it interchanges $c_3a_1$ with $c_3a_2$ and $c_1a_3$ with $c_2a_3$, and it fixes $c_3a_3$ (the transposition $b_1\!\leftrightarrow\!b_2$ acts trivially on direct covers). Here $\beta_2=1$ and $\beta_1=2+2+\delta-6=\delta-2$, where $\delta=|\Delta|$. For $\delta=3$, the two constraint sets share only $c_3a_3$, so $\Delta$ consists of $c_3a_3$, one of $c_3a_1,c_3a_2$, and one of $c_1a_3,c_2a_3$; the four such sets form one orbit, giving a single self-dual core $\whe S^2\vee S^1$, which is not a minimal finite model. For $\delta=4$, omitting any one of the five pairs leaves both constraints satisfied, so $\Delta$ is the full set minus one pair; the five omissions form exactly three orbits, $\{c_3a_1,c_3a_2\}$, $\{c_1a_3,c_2a_3\}$ and $\{c_3a_3\}$, giving three pairwise non-homeomorphic cores, each $\whe S^1\vee S^1\vee S^2$ (Fig.~\ref{k}(c$^\ast$), (c) and (b), respectively; the last one is self-dual). For $\delta=5$, all five pairs are present, giving one core $\whe S^1\vee S^1\vee S^1\vee S^2$ (Fig.~\ref{q}), which is self-dual.
	\end{enumerate}
	
	\smallskip
	\noindent\emph{Case table.} Table~\ref{tab:323} lists one orbit representative per core, by ordered pattern-pair $(\mathsf{P},\mathsf{Q})$, where $\mathsf{P}$ is the pattern of $(C_1,C_2)$ and $\mathsf{Q}$ that of $(A_1,A_2)$; the dual of the row $(\mathsf{P},\mathsf{Q})$ is a row $(\mathsf{Q},\mathsf{P})$ of the same table, so each dual pair appears twice and each self-dual core once (marked $\ast$), and the $20$ rows are the $20$ cores. Every unmarked core is weak homotopy equivalent to a space that has a finite model with fewer than eight points.
	\begin{table}[htbp]
		\centering
		\caption{The $20$ connected eight-point cores of height $2$ with level vector $(3,2,3)$, listed by pattern-pair. A subset is abbreviated by its indices, so that $12$ denotes $\{c_1,c_2\}$ or $\{a_1,a_2\}$ according to the level, and $c_ia_j$ denotes the direct cover $c_i\prec a_j$. The mark $\min$ indicates a minimal finite model of one of the six spaces of Theorem~\ref{thm:eight}, and ``$\ast$'' marks a self-dual core. Of the $20$ cores, $11$ are marked $\min$ and $9$ are not; the pattern-pairs $(\mathsf{F},\mathsf{M})$, $(\mathsf{F},\mathsf{Z})$, $(\mathsf{M},\mathsf{Z})$ and their duals give no core}
		\label{tab:323}
		{\small\setlength{\arraycolsep}{3pt}\renewcommand{\arraystretch}{1.1}
		$\begin{array}{c c c l c c l c}
			\toprule
			(\mathsf{P},\mathsf{Q}) & (C_1,C_2) & (A_1,A_2) & \Delta & \delta & (\beta_1,\beta_2) & & \\ \midrule
		(\mathsf{F},\mathsf{F}) & (123,123) & (123,123) & \varnothing & 0 & (0,4) & \min & \ast\\
		(\mathsf{E},\mathsf{F}) & (12,12) & (123,123) & c_3a_1,\,c_3a_2 & 2 & (1,2) & \min & \\
		(\mathsf{E},\mathsf{F}) & (12,12) & (123,123) & c_3a_1,\,c_3a_2,\,c_3a_3 & 3 & (2,2) & \min & \\
		(\mathsf{F},\mathsf{E}) & (123,123) & (12,12) & c_1a_3,\,c_2a_3 & 2 & (1,2) & \min & \\
		(\mathsf{F},\mathsf{E}) & (123,123) & (12,12) & c_1a_3,\,c_2a_3,\,c_3a_3 & 3 & (2,2) & \min & \\
		(\mathsf{E},\mathsf{E}) & (12,12) & (12,12) & c_1a_3,\,c_3a_1,\,c_3a_3 & 3 & (1,1) & & \ast\\
		(\mathsf{E},\mathsf{E}) & (12,12) & (12,12) & c_1a_3,\,c_2a_3,\,c_3a_1,\,c_3a_2 & 4 & (2,1) & \min & \ast\\
		(\mathsf{E},\mathsf{E}) & (12,12) & (12,12) & c_1a_3,\,c_2a_3,\,c_3a_1,\,c_3a_3 & 4 & (2,1) & \min & \\
		(\mathsf{E},\mathsf{E}) & (12,12) & (12,12) & c_1a_3,\,c_3a_1,\,c_3a_2,\,c_3a_3 & 4 & (2,1) & \min & \\
		(\mathsf{E},\mathsf{E}) & (12,12) & (12,12) & c_1a_3,\,c_2a_3,\,c_3a_1,\,c_3a_2,\,c_3a_3 & 5 & (3,1) & \min & \ast\\
		(\mathsf{E},\mathsf{M}) & (12,12) & (12,123) & c_3a_1,\,c_3a_3 & 2 & (1,1) & & \\
		(\mathsf{E},\mathsf{M}) & (12,12) & (12,123) & c_3a_1,\,c_3a_2,\,c_3a_3 & 3 & (2,1) & \min & \\
		(\mathsf{M},\mathsf{E}) & (12,123) & (12,12) & c_1a_3,\,c_3a_3 & 2 & (1,1) & & \\
		(\mathsf{M},\mathsf{E}) & (12,123) & (12,12) & c_1a_3,\,c_2a_3,\,c_3a_3 & 3 & (2,1) & \min & \\
		(\mathsf{M},\mathsf{M}) & (123,12) & (12,123) & c_3a_3 & 1 & (1,1) & & \ast\\
		(\mathsf{Z},\mathsf{E}) & (12,13) & (12,12) & c_2a_3,\,c_3a_3 & 2 & (1,0) & & \\
		(\mathsf{Z},\mathsf{E}) & (12,13) & (12,12) & c_1a_3,\,c_2a_3,\,c_3a_3 & 3 & (2,0) & & \\
		(\mathsf{E},\mathsf{Z}) & (12,12) & (12,13) & c_3a_2,\,c_3a_3 & 2 & (1,0) & & \\
		(\mathsf{E},\mathsf{Z}) & (12,12) & (12,13) & c_3a_1,\,c_3a_2,\,c_3a_3 & 3 & (2,0) & & \\
		(\mathsf{Z},\mathsf{Z}) & (12,13) & (12,13) & c_2a_3,\,c_3a_2 & 2 & (2,0) & & \ast\\
			\bottomrule
		\end{array}$}
	\end{table}
	
	The ten unordered pattern-pairs exhaust all possibilities for $(C_1,C_2)$ and $(A_1,A_2)$, and within each the admissible $\Delta$ were enumerated in full, so every connected $(3,2,3)$ core is homeomorphic to one in Table~\ref{tab:323}. The $20$ rows are moreover pairwise non-homeomorphic. Indeed, a level-preserving isomorphism carries $(C_1,C_2)$ either to $(C_1,C_2)$ or to $(C_2,C_1)$, and likewise for $(A_1,A_2)$; since each of the four patterns is symmetric in the two sets, $\mathsf{P}$ and $\mathsf{Q}$ are invariants of the core, so rows with different pattern-pairs are non-homeomorphic. Rows sharing a pattern-pair display sets $\Delta$ lying in distinct orbits of the stabiliser of the chosen representatives $(C_1,C_2,A_1,A_2)$, hence are non-homeomorphic by the determination principle. Therefore these $20$ cores are exactly the connected eight-point cores of height $2$ with $\#B_X=2$ and $\#\mathrm{mxl}(X)=\#\mathrm{mnl}(X)=3$, up to homeomorphism.
	
	\noindent Of these, the minimal finite models of the six spaces are: one of $S^2\vee S^2\vee S^2\vee S^2$ (item~1), two of $S^1\vee S^2\vee S^2$ and two of $S^1\vee S^1\vee S^2\vee S^2$ (item~3), $2+3=5$ of $S^1\vee S^1\vee S^2$ (items~4 and~6), and one of $S^1\vee S^1\vee S^1\vee S^2$ (item~6). The remaining nine cores are weak homotopy equivalent to $S^1$, $S^1\vee S^1$ or $S^2\vee S^1$, each of which has a minimal finite model with fewer than eight points (Corollary~\ref{cor:barmak82} and Theorem~\ref{thm:seven}); they are the rows of Table~\ref{tab:323} left unmarked, they are not minimal finite models, and they form Block~II of Table~\ref{tab:nonmin}.
	
	\medskip
	\noindent\textbf{Case 3: $\#B_X=3$.}
	
	In this case $\#\mathrm{mxl}(X)=2$ and $\#\mathrm{mnl}(X)=3$.
	Let
	\[
	\mathrm{mxl}(X)=\{a_1,a_2\}, \quad B_X=\{b_1,b_2,b_3\}, \quad \mathrm{mnl}(X)=\{c_1,c_2,c_3\}.
	\]
	By Lemma~\ref{lem:co}, each $b_i$ lies below both maximal points, so Lemma~\ref{lem:twomax} applies with $m=3$ and $q=3$. Let $s$ be the number of inner minimal points and $G$ the corresponding graph, with $E$ edges and $\kappa$ components; recall that $G$ has minimal degree $\ge2$ on both sides. Since each intermediate point lies above at least two of the $s$ inner minimal points, $s\ge2$. Lemma~\ref{lem:twomax} gives $\beta_1=2-s+\kappa$ and $\beta_2=E-s-3+\kappa$.
	
	\emph{If $s=2$:} each of the three intermediate points lies above both inner minimal points, so $G=K_{2,3}$, $E=6$ and $\kappa=1$. Thus $(\beta_1,\beta_2)=(1,2)$ and $X\whe S^{1}\vee S^{2}\vee S^{2}$ (Fig.~\ref{lprime}(a); its dual is $(a^{\ast})$). This is the unique core with $s=2$.
	
	\emph{If $s=3$:} there are no outer minimal points. The graph $G$ is connected: since it has minimal degree $\ge2$ on both sides, each component contains at least two minimal and two intermediate points, so two components would need at least four of each, but only three of each are available. Hence $\kappa=1$, $\beta_1=0$ and $\beta_2=E-5$. Here $G$ is a bipartite graph on $3+3$ vertices with minimal degree $\ge2$, so $6\le E\le 9$.
	If $E=6$, then $G$ is a $6$-cycle, $\beta_2=1$ and $X\whe S^{2}$, not a minimal finite model. If $E=7$, then $\beta_2=2$ and $X\whe S^{2}\vee S^{2}$, not a minimal finite model. If $E=8$, then $G$ is $K_{3,3}$ minus an edge, $\beta_2=3$ and $X\whe S^{2}\vee S^{2}\vee S^{2}$ (Fig.~\ref{hprime}(c); its dual is $(c^{\ast})$). If $E=9$, then $G=K_{3,3}$, that is, $X=T(2,3,3)$, $\beta_2=4$ and $X\whe S^{2}\vee S^{2}\vee S^{2}\vee S^{2}$ (Fig.~\ref{p}(a); its dual is $(a^{\ast})$), in agreement with Lemma~\ref{lem:tripartite}.
	For each $E\in\{6,7,8,9\}$, the bipartite graph on $3+3$ vertices with minimal degree $\ge2$ and $E$ edges is unique up to isomorphism: it is obtained from $K_{3,3}$ by removing three pairwise disjoint edges, two disjoint edges, one edge, or no edge, respectively. (Removing two edges with a common vertex would leave a vertex of degree $1$.) So each case above gives a single core up to homeomorphism. Thus Case~3 contributes one minimal finite model each of $S^{1}\vee S^{2}\vee S^{2}$, $S^{2}\vee S^{2}\vee S^{2}$ and $S^{2}\vee S^{2}\vee S^{2}\vee S^{2}$, together with their duals.
	
	\medskip
	\noindent\textbf{Case 4: $\#B_X=4$.}
	
	Here $\#\mathrm{mxl}(X)=\#\mathrm{mnl}(X)=2$.
	Let
	\[
	\mathrm{mxl}(X)=\{a_1,a_2\}, \quad B_X=\{b_1,b_2,b_3,b_4\}, \quad \mathrm{mnl}(X)=\{c_1,c_2\}.
	\]
	By Lemma~\ref{lem:crit}(1), each $b_i$ covers at least two of the two minimal points and is covered by at least two of the two maximal points, so it covers both minimal points and is covered by both maximal points. Therefore $X$ is the complete tripartite poset $T(2,4,2)$, and this is the unique configuration up to homeomorphism. By Lemma~\ref{lem:tripartite}, $X\whe\bigvee_{(1)(3)(1)}S^{2}=S^{2}\vee S^{2}\vee S^{2}$ (Fig.~\ref{hprime}(b); self-dual).
	
	\medskip
	Combining all the above cases, we have classified all connected eight-point cores of height $2$. Those that are not minimal finite models of one of the six spaces are listed in Table~\ref{tab:nonmin}; their weak homotopy types are $\bigvee_k S^1$ ($1\le k\le5$), $S^2$, $S^2\vee S^1$ and $S^2\vee S^2$. All of these types have finite models with fewer than eight points: $S^2$ by Theorem~\ref{thm:barmak8}, $\bigvee_k S^1$ with $k\le5$ by Corollary~\ref{cor:barmak82}, and $S^2\vee S^1$ and $S^2\vee S^2$ by Theorem~\ref{thm:seven}. Consequently, none of the cores in Table~\ref{tab:nonmin} is a minimal finite model.
	
	The dual of a minimal finite model is again a minimal finite model of the same space. We conclude that, up to homeomorphism, the eight-point minimal finite models of $S^{1}\vee S^{1}\vee S^{2}$, $S^{1}\vee S^{1}\vee S^{1}\vee S^{2}$, $S^{1}\vee S^{2}\vee S^{2}$, $S^{2}\vee S^{2}\vee S^{2}$, $S^{2}\vee S^{2}\vee S^{2}\vee S^{2}$ and $S^{1}\vee S^{1}\vee S^{2}\vee S^{2}$ are precisely those displayed in Figures~\ref{k},~\ref{q},~\ref{lprime},~\ref{hprime},~\ref{p} and~\ref{twotwo}, respectively. In particular each of these spaces has eight-point models, so by Lemma~\ref{lem:lb8} these are all of its minimal finite models.
	
	Tallying the cores produced above across the level vectors $(p,m,q)$ (each counted up to homeomorphism, with the duals adjoined), the numbers of minimal finite models are:
	\[
	\begin{tabular}{@{}l p{0.68\textwidth}@{}}
		$S^{1}\vee S^{1}\vee S^{2}$ &
		$7$ (5 from $(3,2,3)$; one each from $(2,2,4)$ and $(4,2,2)$),\\[0.3ex]
		
		$S^{1}\vee S^{1}\vee S^{1}\vee S^{2}$ &
		$1$ (from $(3,2,3)$),\\[0.3ex]
		
		$S^{1}\vee S^{2}\vee S^{2}$ &
		$6$ (2 from $(3,2,3)$; one each from $(2,2,4)$, $(4,2,2)$,
		$(2,3,3)$, and $(3,3,2)$),\\[0.3ex]
		
		$S^{2}\vee S^{2}\vee S^{2}$ &
		$5$ (one each from $(2,2,4)$, $(4,2,2)$, $(2,4,2)$,
		$(2,3,3)$, and $(3,3,2)$),\\[0.3ex]
		
		$S^{2}\vee S^{2}\vee S^{2}\vee S^{2}$ &
		$3$ (one each from $(2,3,3)$, $(3,2,3)$, and $(3,3,2)$),\\[0.3ex]
		
		$S^{1}\vee S^{1}\vee S^{2}\vee S^{2}$ &
		$2$ (from $(3,2,3)$),
	\end{tabular}
	\]
	in agreement with Figures~\ref{k},~\ref{q},~\ref{lprime},~\ref{hprime},~\ref{p} and~\ref{twotwo}.
\end{proof}

\begin{corollary}[All connected eight-point cores of height $2$]\label{cor:allcores}
	Up to homeomorphism there are exactly $53$ connected eight-point cores of height $2$. No level vector other than the eight listed below occurs, and the numbers of cores are
	\[
	\setlength{\arraycolsep}{3pt}\renewcommand{\arraystretch}{1.15}
	\begin{array}{c|cccccccc|c}
		(p,m,q) & (3,1,4) & (4,1,3) & (2,2,4) & (4,2,2) & (3,2,3) & (2,3,3) & (3,3,2) & (2,4,2) & \text{total}\\ \hline
		\#\text{cores} & 8 & 8 & 3 & 3 & 20 & 5 & 5 & 1 & 53
	\end{array}
	\]
	Of these, $24$ are minimal finite models, namely those of Theorem~\ref{thm:eight}, and the other $29$ are not.
\end{corollary}

\begin{proof}
	Since $p,q\ge2$, $m\ge1$ and $p+m+q=8$, the possible level vectors are $(2,1,5)$, $(2,2,4)$, $(2,3,3)$, $(2,4,2)$, $(3,1,4)$, $(3,2,3)$, $(3,3,2)$, $(4,1,3)$, $(4,2,2)$ and $(5,1,2)$. In the proof of Theorem~\ref{thm:eight}, Case~1 shows that $(2,1,5)$ admits no core and produces $8$ cores with level vector $(3,1,4)$; Subcase~2.1 produces $3$ with $(2,2,4)$; Subcase~2.2 produces $20$ with $(3,2,3)$ (Table~\ref{tab:323}); Case~3 produces $5$ with $(2,3,3)$; and Case~4 produces the single core $T(2,4,2)$. Applying the duality of Remark~\ref{rem:count} to the cores with $p<q$ adds $8$, $3$ and $5$ cores with level vectors $(4,1,3)$, $(4,2,2)$ and $(3,3,2)$, and shows that $(5,1,2)$ admits no core; the level vectors $(3,2,3)$ and $(2,4,2)$ are self-dual, so the counts $20$ and $1$ already include the duals. Summing gives $8+8+3+3+20+5+5+1=53$, of which $24$ are the minimal finite models of Theorem~\ref{thm:eight}.
\end{proof}

\begin{remark}[The non-minimal eight-point cores]\label{rem:nonmin8}
	Besides the $24$ minimal finite models of Theorem~\ref{thm:eight}, the $53$ cores of Corollary~\ref{cor:allcores} are $16$ cores with a single intermediate point, nine with two intermediate points and four with three, all listed in Table~\ref{tab:nonmin} of Appendix~\ref{app:nonmin}. None of these $29$ cores is a minimal finite model.
\end{remark}

\subsection{Spaces whose minimal finite models have at most eight points}\label{subsec:small}

Combining the results of this paper with the classical results on spheres and wedges of circles, we obtain a complete list, up to weak homotopy equivalence, of the spaces whose minimal finite models have at most eight points; this proves Theorem~\ref{thm:B}.

\begin{corollary}\label{cor:small}
	Suppose that the minimal finite models of a path-connected space $Y$ have $n\le8$ points. Then $n\in\{1,4,5,6,7,8\}$, and:
	\begin{enumerate}
		\item $n=1$ if and only if $Y$ is weakly contractible;
		\item $n=4$ if and only if $Y\whe S^1$;
		\item $n=5$ if and only if $Y\whe S^1\vee S^1$;
		\item $n=6$ if and only if $Y\whe S^2$, $\bigvee_3S^1$ or $\bigvee_4S^1$;
		\item $n=7$ if and only if $Y\whe S^2\vee S^1$, $S^2\vee S^2$, $\bigvee_5S^1$ or $\bigvee_6S^1$;
		\item $n=8$ if and only if $Y\whe S^3$, $\bigvee_7S^1$, $\bigvee_8S^1$, $\bigvee_9S^1$, or one of the six eight-point wedges of Theorem~\ref{thm:A}.
	\end{enumerate}
\end{corollary}

\begin{proof}
	Each listed space has minimal finite models of the stated size: the point for weakly contractible spaces~\cite[Sec.~3.1]{Barmak(2011)}, Theorem~\ref{thm:barmak8} for $S^1$, $S^2$ and $S^3$, Corollary~\ref{cor:barmak82} for wedges of circles, Theorem~\ref{thm:seven} for $S^2\vee S^1$ and $S^2\vee S^2$, and Theorem~\ref{thm:eight} for the six wedges. Since the listed spaces have pairwise different homology (Remark~\ref{rem:wedgeinv}), it remains to show that every $Y$ as in the statement is weak homotopy equivalent to a listed space with the corresponding value of $n$.
	
	Let $X$ be a minimal finite model of $Y$. Then $X$ is a minimal finite space~\cite[Rem.~3.1.2]{Barmak(2011)}, and it is connected because $Y$ is. If $X$ is a point, then $n=1$ and $Y$ is weakly contractible. Otherwise $h=\mathrm{h}(X)\ge1$, and Theorem~\ref{thm:barmak7} gives $4\le 2h+2\le n\le 8$; in particular $n\notin\{2,3\}$ and $h\le3$.
	If $h=1$, then $\mathcal{K}(X)$ is a connected graph, so $Y\whe\bigvee_kS^1$ for some $k\ge0$; and $k\ge1$, since otherwise $Y$ would be weakly contractible and its minimal finite model would be a point. By Corollary~\ref{cor:barmak82} and the values listed after it, $n=4,5,6,7,8$ according as $k=1$, $k=2$, $k\in\{3,4\}$, $k\in\{5,6\}$ or $k\in\{7,8,9\}$. If $h=3$, then $n=8=2h+2$ and $X\cong\mathbb{S}^3(S^0)$ by Theorem~\ref{thm:barmak7}, so $Y\whe S^3$. If $h=2$ and $n\le6$, then $n=6$ and $Y\whe S^2$ by Lemma~\ref{lem:lb7}(1).
	
	Finally, if $h=2$ and $n\in\{7,8\}$, then $X$ is one of the connected cores of height $2$ classified in Theorem~\ref{thm:seven} or in the proof of Theorem~\ref{thm:eight} (Corollary~\ref{cor:allcores}). Hence $Y$ is weak homotopy equivalent to a wedge of circles, to $S^2$, to $S^2\vee S^1$ or $S^2\vee S^2$, or (if $n=8$) to one of the six wedges. A wedge of $k\ge1$ circles is impossible, since its minimal finite models have height $1$ (Theorem~\ref{thm:barmak81}); so is a weakly contractible space, whose minimal finite model is a point; and so is $S^2$, since its minimal finite models have six points. If $n=8$, then $S^2\vee S^1$ and $S^2\vee S^2$ are also impossible, since their minimal finite models have seven points. So $Y$ is weak homotopy equivalent to $S^2\vee S^1$ or $S^2\vee S^2$ if $n=7$, and to one of the six wedges if $n=8$. This covers all cases.
\end{proof}

\begin{remark}\label{rem:nonhomog}
	Call a finite $T_0$-space \emph{homogeneous} (or \emph{pure}) of height $d$ if every maximal chain has length $d$, equivalently if every maximal simplex of $\mathcal{K}(X)$ has dimension $d$. Inspection of the classification (Theorems~\ref{thm:seven} and~\ref{thm:eight}) shows that every minimal finite model of $S^2\vee S^1$, $S^1\vee S^1\vee S^2$, $S^1\vee S^1\vee S^1\vee S^2$, $S^1\vee S^2\vee S^2$ or $S^1\vee S^1\vee S^2\vee S^2$ is non-homogeneous, having maximal chains of length $1$, given by direct covers, as well as of length $2$; whereas every minimal finite model of $S^2$~\cite{Barmak(2011)}, $S^2\vee S^2$, $S^2\vee S^2\vee S^2$ or $S^2\vee S^2\vee S^2\vee S^2$ is homogeneous of height $2$. We do not know whether, more generally, minimal finite models of wedges of spheres of different dimensions are always non-homogeneous.
\end{remark}

\begin{remark}[Related work]\label{rem:ML}
	The preprint~\cite{MosqueraLois(2026)} of Mosquera-Lois cites the earlier arXiv version of this paper. We record the two points of contact.
	\begin{enumerate}
		\item For every $n\ge2$, Mosquera-Lois shows that every finite model of $S^n\vee S^{n-1}\vee S^{n-1}$ has at least $2n+4$ points and that this bound is attained~\cite[Cor.~3.6 and Thm.~4.1]{MosqueraLois(2026)}. For $n=2$ this is the statement of Theorem~\ref{thm:eight} that the minimal finite models of $S^1\vee S^1\vee S^2$ have exactly eight points, which he cites from that version; his argument is independent of ours and works in all dimensions. The case $n=2$ of his classification of connected finite spaces with at most $2n+3$ points and nonzero $n$-th homology over a field~\cite[Thm.~3.4]{MosqueraLois(2026)} likewise recovers the weak homotopy types with $H_2\neq0$ found in Lemma~\ref{lem:lb7}(1) and Theorem~\ref{thm:seven}, and hence the lower bound of Lemma~\ref{lem:lb8}.
		\item The eight-point posets $P$ and $Q$ of~\cite[Sec.~2.2]{MosqueraLois(2026)} are, up to relabelling, the minimal finite models (c) and (d) of Fig.~\ref{k}. With the labels of Fig.~\ref{k}, the order of (d) is obtained from that of (c) by adding the relations $b_2<a_3$ and $c_1<a_3$, so the identity map from (c) to (d) is continuous. Mosquera-Lois proves that this map is a weak homotopy equivalence but not a homotopy equivalence~\cite[Cor.~2.4]{MosqueraLois(2026)}, and obtains such examples in every dimension $n\ge2$ by iterated non-Hausdorff suspension~\cite[Thm.~4.1]{MosqueraLois(2026)}. This answers in the negative Barmak's question whether every weak homotopy equivalence between minimal finite models is a homeomorphism~\cite[Open Problem after Cor.~3.4.8]{Barmak(2011)}.
	\end{enumerate}
	The classification up to homeomorphism and the counts obtained in this paper are not considered in~\cite{MosqueraLois(2026)}.
\end{remark}

\section{Conclusion and future work}\label{sec:concl}

In this paper we classified, up to homeomorphism, the minimal finite models of several wedge sums of one- and two-dimensional spheres, namely $S^{2}\vee S^{1}$ and $S^{2}\vee S^{2}$ on seven points, and $S^{1}\vee S^{1}\vee S^{2}$, $S^{1}\vee S^{1}\vee S^{1}\vee S^{2}$, $S^{1}\vee S^{2}\vee S^{2}$, $S^{2}\vee S^{2}\vee S^{2}$, $S^{2}\vee S^{2}\vee S^{2}\vee S^{2}$, and $S^{1}\vee S^{1}\vee S^{2}\vee S^{2}$ on eight points; a model and its opposite are counted separately whenever they are not homeomorphic. The finite-space part of the arguments is combinatorial: each model is a height-$2$ core, and the cores are found by a direct analysis of covering relations. Their weak homotopy types are identified by McCord's theorem together with the structural lemmas of Section~\ref{sec:prelim}: collapse to a graph, the suspension produced by a union of two cones, and the closed-form Betti-number formulas these yield. These lemmas rest only on elementary cell-attachment arguments for CW complexes.

The classification makes explicit how the combinatorics of a finite poset---the sizes of its levels, the incidences across the middle level, and the absence of beat points---determine the homotopy type of its order complex. Together with the classical results on spheres and wedges of circles, it determines, up to weak homotopy equivalence, all path-connected spaces whose minimal finite models have at most eight points (Corollary~\ref{cor:small}).

\subsection*{Connections with digital topology and discrete Morse theory}

Finite $T_0$-spaces also play a role in digital topology. Khalimsky, Kopperman and Meyer studied connected topologies on finite ordered sets as models for computer graphics~\cite{KKM(1990)}, and Kovalevsky proposed abstract cellular complexes, finite sets of cells partially ordered by a bounding relation and carrying the associated finite topology, as a topological basis for image analysis; he also showed that every finite $T_0$-space is such a complex~\cite{Kovalevsky(1989)}. From this point of view a minimal finite model represents a weak homotopy type with as few cells as possible. Theorem~\ref{thm:A} gives, for each of the eight wedges, the least number of points of a finite $T_0$-space with its weak homotopy type and lists all such spaces; by Remark~\ref{rem:facepos}, the same number bounds from below the number of cells of any regular CW complex weak homotopy equivalent to that wedge.

The arguments of Section~\ref{sec:prelim} also fit into Forman's discrete Morse theory~\cite{Forman(1998)}. In the proof of Lemma~\ref{lem:onemid}, pairing each edge $\{c,a\}$ with the triangle $\{c,b,a\}$ defines a discrete vector field on $\mathcal{K}(X)$ in the sense of Forman. It has no nontrivial closed paths, since no edge containing $b$ is paired, so it is the gradient vector field of a discrete Morse function whose critical cells are exactly the cells of the graph $G$. Thus the conclusion $\mathcal{K}(X)\simeq G$ is an instance of the basic theorem of discrete Morse theory. Minian developed a discrete Morse theory for posets~\cite{Minian(2012)} that extends Forman's theory from regular CW complexes; it applies, for example, to the face posets of the regular CW complexes constructed in Section~\ref{subsec:cw}. It would be interesting to study the minimal finite models classified here from this point of view.

\subsection*{Future work}

Several directions suggest themselves. First, the connected cores with nine and ten points remain to be classified: by Corollary~\ref{cor:small}, spaces such as $S^{1}\vee S^{2}\vee S^{2}\vee S^{2}$ and $S^{2}\vee S^{2}\vee S^{2}\vee S^{2}\vee S^{2}$ have no finite model with fewer than nine points, so their minimal finite models can only be found among these larger cores (both do have models with at most ten points: by Lemmas~\ref{lem:crit} and~\ref{lem:twomax}, adding to $T(2,2,4)$ a minimal point covered by both maximal points gives a nine-point core that models $S^{1}\vee S^{2}\vee S^{2}\vee S^{2}$, and by Lemma~\ref{lem:tripartite}, $T(2,2,6)$ is a ten-point model of $S^{2}\vee S^{2}\vee S^{2}\vee S^{2}\vee S^{2}$). Second, for the family $S^{n}\vee S^{n-1}\vee S^{n-1}$, whose case $n=2$ is $S^1\vee S^1\vee S^2$, the minimal number of points is now known to be $2n+4$ in every dimension, and a connected finite space with at most $2n+3$ points and nonzero $n$-th homology over a field is known to be weak homotopy equivalent to $S^{n}$ or to $S^{n}\vee S^{k}$ for some $1\le k\le n$~\cite[Thm.~3.4]{MosqueraLois(2026)}; classifying minimal finite models up to homeomorphism in dimensions above two, as done here in dimension two, remains open, as does the case of spaces combining different homotopy types, such as $\mathbb{R}P^{2}\vee S^{2}$ and $T^{2}\vee S^{1}$, and of surfaces of higher genus. Finally, it would be natural to seek general lower bounds and uniqueness criteria for wedge sums, extending the $2n+2$ bound known for $S^{n}$, and to determine which pairs of the minimal finite models classified here are joined by weak homotopy equivalences: \cite[Sec.~5]{MosqueraLois(2026)} gives a cohomological condition on the target under which such a map must be a homeomorphism, and Remark~\ref{rem:ML} shows that some such condition is necessary.

\section*{Acknowledgements}
The authors thank the anonymous referee for a careful reading of the manuscript and for valuable comments and suggestions, which have improved the paper.

\appendix
\setcounter{table}{0}\renewcommand{\thetable}{A.\arabic{table}}\renewcommand{\theHtable}{A.\arabic{table}}%

\section{The non-minimal eight-point cores}\label{app:nonmin}

For completeness, Table~\ref{tab:nonmin} lists the $29$ connected eight-point cores of height $2$ that are not minimal finite models (Remark~\ref{rem:nonmin8}). By the determination principle (Lemma~\ref{lem:determination}) each of them is determined up to homeomorphism by the data recorded there; that the table is complete follows from the proof of Theorem~\ref{thm:eight} (see Corollary~\ref{cor:allcores}). Throughout we write $\mathrm{mxl}(X)=\{a_1,\dots\}$, $B_X=\{b_1,\dots\}$, $\mathrm{mnl}(X)=\{c_1,\dots\}$, $C_i=\hat U_{b_i}\cap\mathrm{mnl}(X)$ and $A_i=\hat F_{b_i}\cap\mathrm{mxl}(X)$, and abbreviate a subset by its indices, so that $12$ denotes $\{a_1,a_2\}$ or $\{c_1,c_2\}$ according to the level. Each line of the table describes one core; its opposite is again such a core and is counted on the same line, so that the three blocks contain $8+8$, $9$ and $2+2$ cores respectively.

\begin{table}[htbp]
	\centering
	\caption{The $29$ connected eight-point cores of height $2$ that are not minimal finite models. ``$\ast$'' marks a self-dual core; every other line is a dual pair of two cores}
	\label{tab:nonmin}
	\small
	\begin{tabular}{llll}
		\toprule
		\multicolumn{4}{l}{\textbf{Block I:} $\#B_X=1$, level vector $(3,1,4)$ and dual; $A_1=\{a_1,a_2\}$}\\
		\midrule
		$C_1$ & $\{N(t)\}_{t\notin C_1}$ & & type\\
		\midrule
		$123$ & $\{12\}$ & & $\bigvee_3S^1$\\
		$123$ & $\{123\}$ & & $\bigvee_4S^1$\\
		$12$ & $\{12,12\}$ & & $\bigvee_3S^1$\\
		$12$ & $\{12,13\}$ & & $\bigvee_3S^1$\\
		$12$ & $\{13,23\}$ & & $\bigvee_3S^1$\\
		$12$ & $\{12,123\}$ & & $\bigvee_4S^1$\\
		$12$ & $\{13,123\}$ & & $\bigvee_4S^1$\\
		$12$ & $\{123,123\}$ & & $\bigvee_5S^1$\\
		\midrule
		\multicolumn{4}{l}{\textbf{Block II:} $\#B_X=2$, level vector $(3,2,3)$}\\
		\midrule
		$(\mathsf{P},\mathsf{Q})$ & $(C_1,C_2)$, $(A_1,A_2)$ & $\Delta$ & type\\
		\midrule
		$(\mathsf{E},\mathsf{E})$ & $(12,12)$, $(12,12)$ & $c_1a_3,\,c_3a_1,\,c_3a_3$ & $S^2\vee S^1$ \quad$\ast$\\
		$(\mathsf{E},\mathsf{M})$ & $(12,12)$, $(12,123)$ & $c_3a_1,\,c_3a_3$ & $S^2\vee S^1$\\
		$(\mathsf{M},\mathsf{M})$ & $(123,12)$, $(12,123)$ & $c_3a_3$ & $S^2\vee S^1$ \quad$\ast$\\
		$(\mathsf{Z},\mathsf{E})$ & $(12,13)$, $(12,12)$ & $c_2a_3,\,c_3a_3$ & $S^1$\\
		$(\mathsf{Z},\mathsf{E})$ & $(12,13)$, $(12,12)$ & $c_1a_3,\,c_2a_3,\,c_3a_3$ & $S^1\vee S^1$\\
		$(\mathsf{Z},\mathsf{Z})$ & $(12,13)$, $(12,13)$ & $c_2a_3,\,c_3a_2$ & $S^1\vee S^1$ \quad$\ast$\\
		\midrule
		\multicolumn{4}{l}{\textbf{Block III:} $\#B_X=3$, level vector $(2,3,3)$ and dual; $s=3$, $A_i=\{a_1,a_2\}$}\\
		\midrule
		$G$ & & & type\\
		\midrule
		$K_{3,3}$ minus a perfect matching & ($E=6$) & & $S^2$\\
		$K_{3,3}$ minus two disjoint edges & ($E=7$) & & $S^2\vee S^2$\\
		\bottomrule
	\end{tabular}
\end{table}

\emph{Block~I} ($\#B_X=1$, level vector $(3,1,4)$, and the duals with level vector $(4,1,3)$) records the eight cores found in Case~1 of the proof of Theorem~\ref{thm:eight}. There $A_1=\{a_1,a_2\}$, every minimal point below $b_1$ is covered directly by $a_3$, and each of the $4-|C_1|$ minimal points $t$ not below $b_1$ is covered directly by a set $N(t)\subseteq\mathrm{mxl}(X)$ with $|N(t)|\ge2$, subject to $a_1,a_2\in\bigcup_tN(t)$; the type is $\bigvee_kS^1$ with $k=\sum_t|N(t)|+2|C_1|-5$.

\emph{Block~II} ($\#B_X=2$, level vector $(3,2,3)$) records the nine cores left unmarked in the case table of Subcase~2.2, given by their pattern-pair and their set $\Delta$ of direct covers.

\emph{Block~III} ($\#B_X=3$, level vector $(2,3,3)$, and the duals with level vector $(3,3,2)$) records the two cores with $s=3$ and $E\in\{6,7\}$ found in Case~3; here $G$ is the bipartite graph of that case, on the three minimal and the three intermediate points.

\end{document}